\documentclass[reqno,11pt]{amsart}
\usepackage{amsfonts}
\usepackage{a4wide}
\usepackage{color}
\usepackage{mathrsfs}
\usepackage{mathtools}
\usepackage{amsmath}
\usepackage{amssymb}
\usepackage{bbm}
\usepackage{esint}
\usepackage{nicefrac}
\numberwithin{equation}{section}
\usepackage[colorlinks,citecolor=green,linkcolor=blue]{hyperref}

\usepackage[latin1]{inputenc}
\input xy
\xyoption{all}

\newtheorem{theorem}{Theorem}[section]

\newtheorem{Th}[theorem]{Theorem}
\newtheorem{Lm}[theorem]{Lemma}
\newtheorem{Prop}[theorem]{Proposition}
\newtheorem{Def}[theorem]{Definition}

\newtheorem{Remark}[theorem]{Remark}
\newtheorem{Problem}[theorem]{Problem}
\DeclareMathOperator*{\essinf}{ess\,inf}

\title{Descriptions of traces of weighted Sobolev spaces to Ahlfors--David regular sets in the case $p=1$}
\author{Alexander I. Tyulenev}
\address{Steklov Mathematical Institute of Russian Academy of Sciences}
\email{tyulenev-math@yandex.ru, tyulenev@mi-ras.ru}
\begin{document}
\allowdisplaybreaks
\begin{abstract}
Given $n \in \mathbb{N}$, an Ahlors--David $n$-regular set $S \subset \mathbb{R}^{n+1}$, and a weight $\gamma$ satisfying the
local Muckenhoupt $A_{1}$-condition, we present a complete intrinsic description of the trace-space $W_{1}^{1}(\mathbb{R}^{n+1},\gamma)|_{S}$
of the weighted first-order Sobolev space $W_{1}^{1}(\mathbb{R}^{n+1},\gamma)$ to $S$. Furthermore, we construct a new family of
nonlinear bounded extension operators acting from $W_{1}^{1}(\mathbb{R}^{n+1},\gamma)|_{S}$ to $W_{1}^{1}(\mathbb{R}^{n+1},\gamma)$.
Finally, we find conditions on $\gamma$ that sufficient for the existence of a bounded linear extension operator from $W_{1}^{1}(\mathbb{R}^{n+1},\gamma)|_{S}$ to $W_{1}^{1}(\mathbb{R}^{n+1},\gamma)$.
\end{abstract}
\maketitle
\tableofcontents
\markright{Traces of Sobolev Spaces}
%

%
\section*{Introduction}
\noindent

The problem of complete intrinsic descriptions of trace-spaces of weighted Sobolev spaces $W_{p}^{1}(\mathbb{R}^{d},\gamma)$, $p \in [1,\infty]$,
to different sets $S \subset \mathbb{R}^{d}$ attracted a lot of attention over the past 60 years. Whereas
the case $p > 1$ has been extensively studied by many mathematicians, there are very few results available in the literature related to
the limiting integrability exponent $p=1$. Furthermore, all the investigations available so far dealt with the cases when either sets $S$ or weights $\gamma$
satisfy extra regularity assumptions. The aim of this paper is to make a first step to the filling in this gap by considering the case when both $S$ and $\gamma$ are complicated.

\subsection{The statement of the problem} In order to pose the trace problem in details, we need to introduce some notation.

By a \textit{weight} on $\mathbb{R}^{d}$, $d \in \mathbb{N}$, we always mean a locally integrable
function $\gamma: \mathbb{R}^{d} \to [-\infty,+\infty]$ which is positive almost everywhere with respect to the classical Lebesgue measure $\mathcal{L}^{d}$
on $\mathbb{R}^{d}$.
Given $d \in \mathbb{N}$, $p \in [1,\infty)$, and a weight $\gamma$ on $\mathbb{R}^{d}$, define the weighted Sobolev space $W_{p}^{1}(\mathbb{R}^{d},\gamma)$
as the linear space consisting of all $F \in W^{1,\rm loc}_{1}(\mathbb{R}^{d})$ satisfying
\begin{equation}
\label{eqq.intr_Sobolev}
\|F\|_{W_{p}^{1}(\mathbb{R}^{d},\gamma)}:=\|\gamma^{\frac{1}{p}}F\|_{L_{p}(\mathbb{R}^{d})}+\|\gamma^{\frac{1}{p}}\nabla F\|_{L_{p}(\mathbb{R}^{d})} < +\infty,
\end{equation}
where $\nabla F$ denotes the distributional gradient of $F$.

Without any additional assumptions on $\gamma$ the study of $W_{p}^{1}(\mathbb{R}^{d},\gamma)$ is very difficult. In this paper
we focus on weights satisfying the \textit{local Muckenhoupt} conditions.
Given $p \in (1,\infty)$, by $A_{p}^{\rm loc}(\mathbb{R}^{d})$ we denote the set of all weights on $\mathbb{R}^{d}$ such that
\begin{equation}
\label{eqq.A_p_weights}
\operatorname{C}_{\gamma,p}:=\sup\limits_{l(Q) \le 1}\Bigl(\frac{1}{(l(Q))^{n}}\int\limits_{Q}\gamma(y)\,dy\Bigr)\Bigl(\frac{1}{(l(Q))^{n}}\int\limits_{Q}\gamma^{-\frac{p'}{p}}(y)\,dy\Bigr)^{\frac{p}{p'}} < +\infty,
\end{equation}
where $Q$ denotes a cube in $\mathbb{R}^{d}$ with sides parallel to the coordinate axis and $l(Q)$ its side length.
By $A_{1}^{\rm loc}(\mathbb{R}^{d})$ we denote the set of all weights on $\mathbb{R}^{d}$ such that
\begin{equation}
\label{eqq.A_1_loc_weights}
\operatorname{C}_{\gamma,1}:=\sup\limits_{l(Q) \le 1}\Bigl(\frac{1}{(l(Q))^{n}}\int\limits_{Q}\gamma(y)\,dy\Bigr)\Bigl(\essinf\limits_{y \in Q}\gamma(y)\Bigr)^{-1} < +\infty.
\end{equation}
The classes $A_{p}^{\rm loc}(\mathbb{R}^{d})$, $p \in [1,\infty)$, have been introduced by V.~Rychkov in \cite{Ry} to develop a fruitful theory of weighted Besov and Lizorkin--Triebel spaces.
Those classes give rise to natural generalizations of the famous Muckenhoupt classes $A_{p}$, $p \in [1,\infty)$, (see ch.5 in \cite{Stein} for the details) and include weights with exponential growth at infinity.

We assume that the reader is familiar with the notion of the Hausdorff measures $\mathcal{H}^{t}$ on $\mathbb{R}^{d}$ (see Section 2.2 for details).
Recall (see for instance Section 4.8 in \cite{Evans}) that, given $F \in W_{1}^{1, \rm loc}(\mathbb{R}^{n+1})$, $n \in \mathbb{N}$, there is a set $E_{F} \subset \mathbb{R}^{n+1}$ with $\mathcal{H}^{n}(E_{F})=0$ such that
\begin{equation}
\label{eqq.sharp_representative}
\overline{F}(x):=\varlimsup\limits_{r \to +0}\fint\limits_{Q_{r}(x)}F(y)\,dy \in \mathbb{R} \quad \hbox{for each} \quad x \in \mathbb{R}^{n+1} \setminus E_{F},
\end{equation}
and, furthermore, each point $x \in \mathbb{R}^{n+1} \setminus E_{F}$ is a Lebesgue point of $\overline{F}$. In what follows, we refer to the function
$\overline{F}$ defined in \eqref{eqq.sharp_representative} as the \textit{sharp representative of} $F$.

Given $S \subset \mathbb{R}^{n+1}$ with $\mathcal{H}^{n}(S) > 0$,
we define for each $F \in W_{1}^{1, \rm loc}(\mathbb{R}^{n+1})$ the sharp trace $F|_{S}$ of $F$ to $S$ as an equivalence class (modulo coincidence $\mathcal{H}^{n}$-a.e. on $S$)
of the pointwise restriction of $\overline{F}$ to $S$. Given a Borel function $f:S \to \mathbb{R}$, we call $F \in W_{1}^{1, \rm loc}(\mathbb{R}^{n+1})$ a sharp extension of $f$ if $F|_{S}=[f]$, where
$[f]$ means an equivalence class of $f$ modulo coincidence $\mathcal{H}^{n}$-a.e. on $S$. In what follows we will always \textit{identify} any Borel function $f$ with its equivalence class $[f]$ on $S$.
We define the trace space $W_{p}^{1}(\mathbb{R}^{n+1},\gamma)|_{S}$  as the linear space of sharp traces
$F|_{S}$ of all elements $F \in W_{p}^{1}(\mathbb{R}^{n+1},\gamma)$. We equip this space with the corresponding
quotient-space norm, that is, given $f \in W_{p}^{1}(\mathbb{R}^{n+1},\gamma)|_{S}$, we put
\begin{equation}
\label{eqq.quotient_trace_norm}
\|f\|_{W_{p}^{1}(\mathbb{R}^{n+1},\gamma)|_{S}}:=\inf\{\|F\|_{W_{p}^{1}(\mathbb{R}^{n+1},\gamma)}:f=F|_{S}\}.
\end{equation}
One can naturally define the trace operator $\operatorname{Tr}_{S}:W_{p}^{1}(\mathbb{R}^{n+1},\gamma) \to W_{p}^{1}(\mathbb{R}^{n+1},\gamma)|_{S}$ which takes
$F$ and returns $F|_{S}$. Finally, we refer to a map $\operatorname{Ext}_{S}:W_{p}^{1}(\mathbb{R}^{n+1},\gamma)|_{S} \to W_{p}^{1}(\mathbb{R}^{n+1},\gamma)$ as an \textit{extension operator for $S$}
if:

\begin{itemize}

\item[\((EXT1)\)] there is $C > 0$ such that $\|\operatorname{Ext}_{S}(f)\|_{W_{p}^{1}(\mathbb{R}^{n+1},\gamma)} \le C\|f\|_{W_{p}^{1}(\mathbb{R}^{n+1},\gamma)|_{S}}$ for all $f \in W_{p}^{1}(\mathbb{R}^{n+1},\gamma)|_{S}$;

\item[\((EXT2)\)] $\operatorname{Tr}_{S} \circ \operatorname{Ext}_{S} = \operatorname{Id}$ on the space $W_{p}^{1}(\mathbb{R}^{n+1},\gamma)|_{S}$.

\end{itemize}


\begin{Problem}
\label{Trace_Problem}
Let $n \in \mathbb{N}$, $p \in [1,\infty)$, and $\gamma \in A_{p}^{\rm loc}(\mathbb{R}^{n+1})$. Let $S \subset \mathbb{R}^{n+1}$ be an arbitrary closed set
with $\mathcal{H}^{n}(S) > 0$.

\begin{itemize}

\item[\((\textbf{Q1})\)] Given a Borel function $f:S \to \mathbb{R}$, find necessary and sufficient conditions
for the existence of a sharp extension $F \in W_{p}^{1}(\mathbb{R}^{n+1},\gamma)$ of the function $f$.

\item[\((\textbf{Q2})\)] Using only the geometry of the set $S$ and the values of a function $f \in W_{p}^{1}(\mathbb{R}^{n+1},\gamma)|_{S}$, compute the norm
$\|f\|_{W_{p}^{1}(\mathbb{R}^{n+1},\gamma)|_{S}}$ up to some universal constants.

\item[\((\textbf{Q3})\)] Does there exist an extension operator $\operatorname{Ext}_{S}:W_{p}^{1}(\mathbb{R}^{n+1},\gamma)|_{S} \to W_{p}^{1}(\mathbb{R}^{n+1},\gamma)$ for $S$?
Can $\operatorname{Ext}_{S}$ be linear?

\end{itemize}
\end{Problem}

At the beginning of June 2016 the author presented his results related to Problem \ref{Trace_Problem} (those results were published in \cite{T1}) at the 9th Whitney Problems Workshop in Haifa.
Right after the talk Michal Wojciechowski posed the following problem.

\begin{Problem}
\label{Michal's_Problem}
Given a weight  $\gamma \in A_{1}^{\rm loc}(\mathbb{R}^{n+1})$, find necessary and sufficient conditions
for the existence of a linear extension operator $\operatorname{Ext}_{\mathbb{R}^{n}}:W^{1}_{1}(\mathbb{R}^{n+1},\gamma)|_{\mathbb{R}^{n}} \to W^{1}_{1}(\mathbb{R}^{n+1},\gamma)$.
\end{Problem}

\subsection{Previously known results}
The literature devoted to Problem A is huge. However, the majority of studies were focused on very special particular cases. Now we briefly
describe only the most powerful results available in the literature, the reader can find all necessary historical background in the corresponding papers mentioned below.

We recall that, given $d \in \mathbb{N}$ and $t \in (0,d]$, a closed set $S \subset \mathbb{R}^{d}$ is said to be Ahlfors--David $t$-regular, if there are
constants $C^{S}_{1},C^{S}_{2} > 0$ such that (for each $x \in \mathbb{R}^{n+1}$ and $r > 0$, we set $Q_{r}(x):=\prod_{i=1}^{n+1}[x_{i}-r,x_{i}+r]$)
\begin{equation}
\label{eqq.Ahlfors_regularity}
C^{S}_{1}r^{t} \le \mathcal{H}^{t}(Q_{r}(x) \cap S) \le C^{S}_{2}r^{t} \quad \hbox{for all} \quad (x,r) \in S\times (0,1].
\end{equation}

\begin{itemize}

\item[\((\textbf{R1})\)] In the special case when $S=\mathbb{R}^{n}\times\{0\}$ and $\gamma$ depends only on the last variable $x_{n+1}$ Problem \ref{Trace_Problem} was completely
solved for $p > 1$ even in the context of more general Liouville-type spaces by G.A.~Kalyabin \cite{K} for the case of periodic functions and then
by B.V.~Tandit \cite{Tan} for the nonperiodic case. Based on that machinery, A.S.~Ginzburg completely solved Problem \ref{Trace_Problem} (under similar assumptions on $\gamma$) in the case $p=1$ \cite{Gin}.

\item[\((\textbf{R2})\)] In the case when $S=\mathbb{R}^{n}\times\{0\}$ and $\gamma \in A_{p}^{\rm loc}(\mathbb{R}^{n+1})$ is a generic weight depending on all variables, Problem \ref{Trace_Problem}
was completely solved by the author \cite{T4,T1}. In \cite{T4} the case $p > 1$ was considered,
the corresponding trace space was characterized in terms of some new modifications of Besov-type spaces of variable smoothness developed by the author. Furthermore,
 a  new linear extension operator was constructed. The more delicate case $p=1$ was studied in \cite{T1} and called for an introduction
of a new function space. Furthermore, a new nonlinear extension operator $\operatorname{Ext}_{S,\gamma}$ was constructed.

\item[\((\textbf{R3})\)] There are results related to the generalization of Problem \ref{Trace_Problem} in the
context of metric measure spaces $\operatorname{X}=(\operatorname{X},\operatorname{d},\mu)$ with doubling measures $\mu$
satisfying the so-called Poncar\'e-type inequalities.
If $S \subset \operatorname{X}$ is lower content regular, $p \in (1,\infty)$, and $\gamma \equiv 1$,
the corresponding analog of Problem \ref{Trace_Problem} was solved by the author in \cite{T5}. The case when $S \subset \operatorname{X}$ is
Ahlfors--David $n$-regular, $p=1$ and $\gamma \equiv 1$ was considered  by L.~Maly in \cite{Maly}.

\end{itemize}

While metric measure spaces considered in \cite{Maly} include finite dimensional Euclidean spaces $\mathbb{R}^{d}$ equipped with a weighted
Lebesgue measure $\gamma\mathcal{L}^{d}$ with $\gamma \in A_{1}^{\rm loc}(\mathbb{R}^{d})$, the results from \cite{Maly} \textit{do not cover} results obtained in \cite{T1}.
The point is that while the hyperplane $\mathbb{R}^{n} \times \{0\}$ is an Ahlfors--David $n$-regular set considered in
$(\mathbb{R}^{n+1},\|\cdot\|,\mathcal{L}^{n+1})$, it can fail to satisfy Ahlfors--David codimension $1$ regularity condition (this concept was used in \cite{Maly}) in metric measure spaces
$(\mathbb{R}^{n+1},\|\cdot\|,\gamma\mathcal{L}^{n+1})$ for some weights $\gamma \in A_{1}^{\rm loc}(\mathbb{R}^{n+1})$. The simplest example for $n=2$ is given by $\gamma(x)=\chi_{B_{1}(0)}(x)\|x\|^{-1}+\chi_{\mathbb{R}^{2}\setminus B_{1}(0)}(x)$, $x \in \mathbb{R}^{2}$.
As a result, the study of Problem \ref{Trace_Problem} for generic sets and weights is an actual complicated problem requiring
new methods and tools. In the present paper we focus on the case $p=1$.

The facts of nonlinearity of the operator $\operatorname{Ext}_{S,\gamma}$ constructed in \cite{T1} and its dependence on $\gamma$ look quite reasonable and reflect the essence of the matter.
Indeed, if $\gamma \equiv 1$ (the nonweighted case), then according to the famous result of J.~Peetre \cite{P}, the linear extension operator
does not exists (see also \cite{KW} for an alternative proof). At the same time, for some special weights $\gamma$ depending only on $x_{n+1}$ and tending to $+\infty$ as $x_{n+1} \to +0$ fast enough it is
possible to construct a linear extension operator \cite{Gin}.

Unfortunately, the methods introduced in \cite{T1} were essentially based on the Euclidean structure of the ``trace set'' $S=\mathbb{R}^{n}\times \{0\} \subset \mathbb{R}^{n+1}$ and
it is difficult to adopt such a machinery to the case of generic sets $S$. Furthermore, the characterization of the trace-space given in \cite{T1} was not fully intrinsic.
Indeed, the expression for the norm of the function in the corresponding function space contains the infimum over infinite number of admissible tilings of $\mathbb{R}^{n}$ which is hardly
ever to calculate in applications. In this paper we eliminate all the mentioned drawbacks by introducing new methods and new extension operator.
Our new characterization will be essentially simpler than that of given in \cite{T1}.

\subsection{Main results} Given a weight $\gamma \in A_{1}^{\rm loc}(\mathbb{R}^{n+1})$, for each $r > 0$ and $x \in \mathbb{R}^{n+1}$, we
put $\underline{\gamma}_{r}(x):=\essinf_{y \in Q_{r}(x)}\gamma(y)$ for brevity.
Given $q > 1$ and a nonempty set $S \subset \mathbb{R}^{n+1}$, for each $k \in \mathbb{Z}$, we put $S^{k}(q):=\{x \in S: \varlimsup\limits_{r \to +0}\underline{\gamma}_{r}(x) > q^{k}\}.$
Finally, given $k \in \mathbb{Z}$ and $x \in S^{k}(q)$, we put $r_{k}(x):=\sup\{r \in (0,1]: \underline{\gamma}_{5r}(x) \geq q^{k}\}.$
For each $k \in \mathbb{Z}$, for every $x \in S^{k}(q)$,
\begin{equation}
\notag
\mathcal{E}^{S}_{k}[f](x):=\inf\limits_{c \in \mathbb{R}}\frac{1}{\mathcal{H}^{n}(Q_{r_{k}}(x) \cap S)}\int\limits_{Q_{r_{k}}(x)\cap S}|f(y)-c|\,d\mathcal{H}^{n}(y)
\end{equation}
and define the \textit{Besov-type functional} by letting, for each $f \in L_{1}^{\rm loc}(\mathcal{H}^{n}\lfloor_{S})$,
\begin{equation}
\notag
\mathcal{BN}_{\gamma,q}[f]:=\sum\limits_{k \in \mathbb{Z}}q^{k}\int\limits_{S^{k}(q)}\mathcal{E}^{S}_{k}[f](x)\,d\mathcal{H}^{n}(x).
\end{equation}
Furthermore, we introduce the \textit{Lebesgue-type functional} by letting, for each $f \in L_{1}^{\rm loc}(\mathcal{H}^{n}\lfloor_{S})$,
\begin{equation}
\notag
\mathcal{LN}_{\gamma}[f]:=\int\limits_{S}\gamma_{1}(x)|f(x)|\,d\mathcal{H}^{n}(x).
\end{equation}

Now we are ready to formulate the \textit{first main} result of the present paper. This gives a complete solution to Problem \ref{Trace_Problem} in the case of
Ahlfors--David regular sets.
\begin{Th}
\label{Th.first_main}
Let $n \in \mathbb{N}$ and $\gamma \in A^{\rm loc}_{1}(\mathbb{R}^{n+1})$. Let $S \subset \mathbb{R}^{n+1}$ be a closed Ahlfors--David $n$-regular set.
Then there is a constant $\underline{q} \geq 1$ such that, for each $q > \underline{q}$, the following holds:

\begin{itemize}

\item[\((\textbf{1})\)] A Borel function $f:S \to \mathbb{R}$ belongs to the trace-space $W_{1}^{1}(\mathbb{R}^{n+1},\gamma)|_{S}$ if and only
if $\mathcal{N}_{\gamma,q}[f]:=\mathcal{BN}_{\gamma,q}[f]+\mathcal{LN}_{\gamma}[f] < +\infty$. Furthermore, there is a constant $C > 0$ depending only on $n$, $\operatorname{C}_{\gamma,1}$, $q$ and
$C^{S}_{1},C^{S}_{2}$ such that
\begin{equation}
\label{eqq.norm_equivalence}
\frac{1}{C}\mathcal{N}_{\gamma,q}[f] \le \|f\|_{W_{1}^{1}(\mathbb{R}^{n+1},\gamma)|_{S}} \le C \mathcal{N}_{\gamma,q}[f].
\end{equation}

\item[\((\textbf{2})\)] There exists a nonlinear extension operator $\operatorname{Ext}_{S,\gamma}:W_{1}^{1}(\mathbb{R}^{n+1},\gamma)|_{S} \to W_{1}^{1}(\mathbb{R}^{n+1},\gamma)$.

\end{itemize}

\end{Th}

The second main result of this paper gives a partial answer to Wojciechowski's question.

\begin{Th}
\label{Th.Michal's_problem}
Let $S \subset \mathbb{R}^{n+1}$ be a closed Ahlfors--David $n$-regular set. If a weight $\gamma \in A_{1}^{\rm loc}(\mathbb{R}^{n+1})$ is such that
\begin{equation}
\label{eqq.criterion}
\varlimsup\limits_{r \to +0}\underline{\gamma}_{r}(x) = +\infty \quad \hbox{for all} \quad x \in S,
\end{equation}
then, there exists a linear extension operator $\operatorname{Ext}_{S,\gamma}: W_{1}^{1}(\mathbb{R}^{n+1},\gamma)|_{S} \to W_{1}^{1}(\mathbb{R}^{n+1},\gamma)$.
\end{Th}

\textbf{Acknowledgements.} I am grateful to the organizers of the special volume dedicated to Professor Hans Triebel for their kind invitation. While I had very few personal
contacts with Professor Tribel, I deeply felt the influence of his ideas to my mathematical formation. I would like to express my special gratitude to Professor Luigi Ambrosio who proposed to use the Eilenberg inequality
for the proof of Proposition \ref{Prop.removability}. Finally, I thank my student Alexey Chikalov, who found some typos and inaccuracies in the preliminary version of the paper.


\section{Preliminaries}

The goal of this section is to recall some background related to geometric measure theory, weighted Sobolev spaces and to set the
terminology that we shall adopt in the paper.

The symbol $C$ always denotes a positive constant which depends only on the fixed parameters and probably on auxiliary functions, unless otherwise stated;
its value may vary from line to
line. The meaning of $A \lesssim B$ is given by that there exists a positive constant $C \in (0,+\infty)$ such that
$A \le C B$. The symbol $A \approx B$ will be used as an abbreviation of $A \lesssim B \lesssim A$.

The symbols $\mathbb{Z}$ and $\mathbb{N}_{0}$ denote the set of all integer numbers
and the set of all nonnegative integer numbers, respectively.
Throughout the paper, given $d \in \mathbb{N}$, the symbol $\mathbb{R}^{d}$ means the linear space of all strings $x=(x_{1},...,x_{d})$ of real numbers equipped
with the $l_{\infty}$-norm, i.e. $\|x\|:=\|x\|_{\infty}:=\max\{|x_{1}|,...,|x_{d}|\}$. Given a set $E \subset \mathbb{R}^{d}$, we denote by
$\operatorname{cl}E$, $\operatorname{int}E$ and $E^{c}$ the closure, the interior, and the complement of $E$ in $\mathbb{R}^{d}$,
respectively. Given a set $E \subset \mathbb{R}^{d}$, we denote by $\chi_{E}$ the characteristic function of
$E$ and by $\#E \in \mathbb{N}_{0} \cup \{+\infty\}$ the cardinality of $E$. Given two nonempty sets $E_{1},E_{2} \in \mathbb{R}^{d}$, the symbol
$\operatorname{dist}(E_{1},E_{2})$ denotes the \textit{distance}  between $E_{1}$ and $E_{2}$ in metric generated by $\|\cdot\|_{\infty}$-norm.
If $S \subset \mathbb{R}^{d}$ is a nonempty set, we define, for each $\delta > 0$, the $\delta$-neighborhood of $S$ by the equality
\begin{equation}
\label{eqq.neighborhood}
U_{\delta}(S):=\{y \in \mathbb{R}^{d}:\operatorname{dist}(x,S) < \delta\}.
\end{equation}

\subsection{Cubes}
Throughout the paper a \textit{cube} $Q$ in $\mathbb{R}^{d}$, $d \in \mathbb{N}$, is a \textit{closed} cube with sides \textit{parallel to the coordinate axis}, i.e. the ball in $l_{\infty}$-norm.
If $Q = \prod_{i=1}^{d}[x_{i}-r,x_{i}+r]$, then we usually write $Q = Q_{r}(x)$ and call $x$ its \textit{center} and $l(Q):=2r$ its \textit{side length}, respectively.
If $Q=Q_{r}(x)$, then for every $c > 0$ we put $cQ:=Q_{cr}(x)$.

\textbf{Important notation.} Throughout the paper,
by a \textit{dyadic cube} in $\mathbb{R}^{d}$ we always mean a cube $Q_{k,m}:=\prod_{i=1}^{d}[\frac{m_{i}}{2^{k}},\frac{m_{i}+1}{2^{k}}]$ with $k \in \mathbb{Z}$ and
$m \in \mathbb{Z}^{d}$. If $S \subset \mathbb{R}^{d}$ is a closed nonempty set and $c > 1$, then, given $k \in \mathbb{Z}$, we put $\mathcal{D}_{k}(S,c):=\{Q_{k,m}:cQ_{k,m} \cap S \neq \emptyset\}.$
Furthermore, $\mathcal{D}(S,c):=\bigcup_{k \in \mathbb{Z}}\mathcal{D}_{k}(S,c), \mathcal{D}_{+}(S,c):=\bigcup_{k \in \mathbb{N}_{0}}\mathcal{D}_{k}(S,c).$
If either $S = \mathbb{R}^{d}$ or $c=1$, then we omit it from the corresponding notation.

\subsection{Some geometric measure theory background}

By a \textit{measure on} $\mathbb{R}^{d}$, $d \in \mathbb{N}$, (or just measure) we always mean an outer measure, i.e. a
monotone countably subadditive function $\mathfrak{m}:2^{\mathbb{R}^{d}} \to [0,+\infty]$ with $\mathfrak{m}(\emptyset) = 0$. We say that a measure $\mathfrak{m}$
is locally finite if $\mathfrak{m}(Q) < +\infty$ for every cube in $\mathbb{R}^{d}$. Given a closed set $S \subset \mathbb{R}^{d}$ and a measure $\mathfrak{m}$ on $\mathbb{R}^{d}$, by
$\mathfrak{m}\lfloor_{S}$ we always denote the \textit{restriction of} $\mathfrak{m}$ to $S$, i.e. a measure defined by the equality $\mathfrak{m}\lfloor_{S}(E):=\mathfrak{m}(S \cap E)$ for 
each set $E \subset \mathbb{R}^{d}$.

A family $\mathcal{F}$ of subset of $\mathbb{R}^{d}$ is said to be \textit{locally finite}, if
for every $x \in \mathbb{R}^{d}$ there is $\delta > 0$ such that
the number of different sets from $\mathcal{F}$ having nonempty intersection with $Q_{\delta}(x)$ is finite.
We say that $\mathcal{F}$ is 
\textit{noneoverlapping} if different sets in $\mathcal{F}$ have disjoint interiors.
Furthermore, $\mathcal{F}$ is said to be \textit{disjoint} if different sets in $\mathcal{F}$ have an empty intersection. Finally, by $\operatorname{MULT}(\mathcal{F})$ we denote the \textit{covering multiplicity} of $\mathcal{F}$,
i.e., the minimal $M' \in \mathbb{N} \cup \{+\infty\}$ such that every point $x \in \mathbb{R}^{d}$ belongs to at most $M'$
sets from $\mathcal{F}$.

Given a Borel nonempty set $S \subset \mathbb{R}^{d}$, the symbol $\mathfrak{B}(S)$ stands for the set of all Borel functions $f:S \to \mathbb{R}$.

\begin{Remark}
\label{Rem.dyadic_lattice_multiplicity}
Recall that all cubes considered in this paper are assumed to be closed. Fix $d \in \mathbb{N}$ and $c \geq 1$. Given $k \in \mathbb{Z}$ and $Q_{1},Q_{2} \in \mathcal{D}_{k}$ in $\mathbb{R}^{d}$, $d \in \mathbb{N}$,
assume that $Q_{1} \cap cQ_{2} \neq \emptyset$ for some $c \geq 1$, then $Q_{1} \cap [c]Q_{2} \neq \emptyset$. Consequently, $Q_{2} \cap [c]Q_{1} \neq \emptyset$ and
$Q_{2} \subset ([c] +2)Q_{1}$. Hence, elementary volume observations lead to
\begin{equation}
\label{eqq.dyadic_multiplicity}
\operatorname{MULT}(\{cQ: Q \in \mathcal{D}_{k}\}) \le ([c]+2)^{d}.
\end{equation}
\end{Remark}

Given a locally finite measure $\mathfrak{m}$ on $\mathbb{R}^{d}$ and a bounded set $G \subset \mathbb{R}^{d}$, for each $f \in L_{1}^{\rm loc}(\mathfrak{m})$,
\begin{equation}
\notag
\fint\limits_{G}f(y)\,d\mathfrak{m}(y):=
\begin{cases}
&\frac{1}{\mathfrak{m}(G)}\int\limits_{G}f(y)\,d\mathfrak{m}(y) \quad \hbox{if} \quad \mathfrak{m}(G) > 0;\\
&0 \quad \hbox{if} \quad  \mathfrak{m}(G)=0.
\end{cases}
\end{equation}
We define
the \textit{averaging map} $M^{\mathfrak{m}}_{G}:=L_{1}^{\rm loc}(\mathfrak{m}) \to \mathbb{R}$, by the equality
\begin{equation}
\label{eqq.maximal_function}
M^{\mathfrak{m}}_{G}[f]:=\fint\limits_{G}f(y)\,d\mathfrak{m}(y), \qquad f \in L_{1}^{\rm loc}(\mathfrak{m}).
\end{equation}
We say that $x \in \operatorname{supp}\mathfrak{m}$ is an $\mathfrak{m}$-\textit{Lebesgue point} of $f \in L_{1}^{\rm loc}(\mathfrak{m})$ if 
\begin{equation}
\lim\limits_{r \to 0}\fint\limits_{B_{r}(x)}|f(x)-f(y)|\,d\mathfrak{m}(y)=0.
\end{equation} 
Furthermore, we introduce \textit{the averaged local best approximation} of $f$ by constants on $G$ by
\begin{equation}
\label{eqq.local_best_approximation}
\mathcal{E}^{\mathfrak{m}}_{G}[f]:=\inf\limits_{c \in \mathbb{R}}\fint\limits_{G}|f(y)-c|\,d\mathfrak{m}(y), \qquad f \in L_{1}^{\rm loc}(\mathfrak{m}).
\end{equation}
\begin{Remark}
\label{Rem.the_best_constant}
It is clear that the infimum in \eqref{eqq.local_best_approximation} is achieved. We denote the corresponding constant by $\operatorname{c}(G,\mathfrak{m},f)$.
\end{Remark}
Finally, \textit{the averaged difference map} $\mathcal{A}^{\mathfrak{m}}_{G}:L_{1}^{\rm loc}(\mathfrak{m}) \to [0,+\infty)$ is given by
\begin{equation}
\notag
\mathcal{A}^{\mathfrak{m}}_{G}[f]:=
\fint\limits_{G}\fint\limits_{G}|f(y)-f(z)|\,d\mathfrak{m}(y)d\mathfrak{m}(z), \qquad f \in L_{1}^{\rm loc}(\mathfrak{m}).
\end{equation}

\begin{Remark}
\label{Rem.different_averagings}
Given $f \in L_{1}^{\rm loc}(\mathfrak{m})$, for each bounded set $G \subset \mathbb{R}^{d}$, we have
\begin{equation}
\label{eqq.different_averagings}
\mathcal{A}^{\mathfrak{m}}_{G}[f] \le 2\mathcal{E}^{\mathfrak{m}}_{G}[f] \le 2\fint\limits_{G}\Bigl|f(y)-\fint\limits_{G}f(z)\,d\mathfrak{m}(z)\Bigr|\,d\mathfrak{m}(y) \le 2\mathcal{A}^{\mathfrak{m}}_{G}[f](x).
\end{equation}
\end{Remark}

The following elementary assertion is a folklore and will be commonly used later. Probably, the first proof (even in a more general context) was given in \cite{Br}.
\begin{Lm}
\label{Lm.subadditivity_best_approximation}
Let $\mathfrak{m}$ be a locally finite measure on $\mathbb{R}^{d}$ and let $\{\Omega_{i}\}_{i=1}^{N}$, $N \in \mathbb{N}$, be a family of bounded subsets of $\mathbb{R}^{d}$ such that
$\mathfrak{m}(\bigcap_{i=1}^{N}\Omega_{i}) \geq \delta\mathfrak{m}(\bigcup_{i=1}^{N}\Omega_{i}) > 0$ for some $\delta > 0$. Then, for every $f \in L_{1}^{\rm loc}(\mathfrak{m})$, the following properties hold
(we set $\Omega=\bigcup_{i=1}^{N}\Omega_{i}$):
\begin{equation}
\label{eqq.2.4'}
\max\limits_{1 \le i \le N}\mathcal{E}^{\mathfrak{m}}_{\Omega_{i}}[f] \le \frac{2}{\delta}\mathcal{E}^{\mathfrak{m}}_{\Omega}[f], \qquad \mathcal{E}^{\mathfrak{m}}_{\Omega}[f] \le \frac{2}{\delta}\sum\limits_{i=1}^{N}\mathcal{E}^{\mathfrak{m}}_{\Omega_{i}}[f].
\end{equation}
\end{Lm}

\begin{proof}
We put $\underline{\Omega}:= \bigcap_{i=1}^{N}\Omega_{i}$ and $\operatorname{c}(\Omega):=\operatorname{c}(\Omega,\mathfrak{m},f)$ for brevity.
We proof only the second inequality in \eqref{eqq.2.4'} since the first inequality can be established by similar arguments.
Taking into account Remarks \ref{Rem.the_best_constant} and \ref{Rem.different_averagings} we have
\begin{equation}
\label{eqq.2.5'}
\begin{split}
&\mathcal{E}^{\mathfrak{m}}_{\Omega}[f]=\fint\limits_{\Omega}|f(x)-\operatorname{c}(\Omega)|\,d\mathfrak{m}(x) \le \fint\limits_{\Omega}\Bigl|f(x)-\fint\limits_{\underline{\Omega}}f(y)\,d\mathfrak{m}(y)\Bigr|\,d\mathfrak{m}(x)\\
&\le \sum\limits_{i=1}^{N}\fint\limits_{\Omega_{i}}\Bigl|f(x)-\fint\limits_{\underline{\Omega}}f(y)\,d\mathfrak{m}(y)\Bigr|\,d\mathfrak{m}(x) \le \frac{1}{\delta}\sum\limits_{i=1}^{N}\mathcal{A}^{\mathfrak{m}}_{\Omega_{i}}[f] \le \frac{2}{\delta}\sum\limits_{i=1}^{N}\mathcal{E}^{\mathfrak{m}}_{\Omega_{i}}[f].
\end{split}
\end{equation}
\end{proof}

Given $d \in \mathbb{N}$, by $\mathcal{L}^{d}$ we denote the classical \textit{Lebesgue measure} on $\mathbb{R}^{d}$. We also recall the classical concept
of the $t$-Hausdorff measure for $t \in (0,d]$. Given a set $E \subset \mathbb{R}^{d}$, for each $\delta > 0$,
$\mathcal{H}^{t}_{\delta}(E):=\inf\sum_{i=1}^{\infty} (\operatorname{diam}G_{i})^{t}$,
where the infimum is taken over all countable coverings $\{G_{i}\}_{i \in \mathcal{I}}$ of $E$
such that $\operatorname{diam}G_{i} < \delta$ for all  $i\in \mathcal{I}$.
Finally, we define the \textit{$t$-Hausdorff measure} of $E$ by the equality
$\mathcal{H}^{t}(E):=\lim_{\delta \to +0}\mathcal{H}^{t}_{\delta}(E).$ Given a closed set $S \subset \mathbb{R}^{d}$, we shall
usually consider the \textit{restriction} of $\mathcal{H}^{t}$ to $S$. This is a measure on $\mathbb{R}^{d}$ which is denoted by $\mathcal{H}^{t}\lfloor_{S}$ and
defined by $\mathcal{H}^{t}\lfloor_{S}(E):=\mathcal{H}^{t}(S \cap E)$ for every set $E \subset \mathbb{R}^{d}$.

\begin{Def}
\label{Def.Ahlfors_David}
Given $d \in \mathbb{N}$ and $t \in (0,d]$, a closed set $S \subset \mathbb{R}^{d}$ is said to be Ahlfors--David $t$-regular if there are $C^{S}_{1},C^{S}_{2} > 0$ such that
\begin{equation}
\label{eqq.Ahlfors_regular}
C^{S}_{1} r^{t} \le \mathcal{H}^{t}(Q_{r}(x)\cap S) \le C^{S}_{2}r^{t} \quad \hbox{for all} \quad (x,r) \in S \times (0,1].
\end{equation}
By $\mathcal{ARD}^{t}(\mathbb{R}^{d})$ we denote the class of all Ahlfors--David $t$-regular sets in $\mathbb{R}^{d}$.
\end{Def}

Now we summarize some elementary geometric properties of the Ahlfors--David regular sets. 
In fact, such properties are folklore and can be found in many papers. 
For example, the interesting
reader can easily deduce them from Proposition 1.3 in \cite{T2}.
\begin{Prop}
\label{Prop.porous_property}
Let $t \in (0,d)$, $S \in \mathcal{ARD}^{t}(\mathbb{R}^{d})$ and $c \geq 1$. Then:
\begin{itemize}

\item[\((\text{Por.1})\)] $\mathcal{H}^{d}(S) = 0$;

\item[\((\text{Por.2})\)] there are parameters $\underline{j}(S,c) \in \mathbb{N}$ and $\varkappa_{S} > 0$ such that, for each $k \in \mathbb{N}$ and $j \geq \underline{j}(S,c)$,
\begin{equation}
\notag
\mathcal{L}^{d}\Bigl(Q \setminus \bigcup_{Q' \in \mathcal{D}_{k+j}(S,c)}cQ'\Bigr) \geq \varkappa_{S} \mathcal{L}^{d}(Q) \quad \hbox{for every} \quad Q \in \mathcal{D}_{k}.
\end{equation}
\end{itemize}
\end{Prop}




\textbf{Important notation.} If $S \in \mathcal{ADR}^{t}(\mathbb{R}^{d})$ for some $t \in (0,d]$ and $\mathfrak{m}=\mathcal{H}^{t}\lfloor_{S}$, then
we set $\mathcal{E}^{S}_{Q}[f]:=\mathcal{E}^{\mathcal{H}^{t}\lfloor_{S}}_{Q}[f]$
and $\mathcal{A}_{Q}^{S}[f]:=\mathcal{A}_{Q}^{\mathcal{H}^{t}\lfloor_{S}}[f]$ for each cube $Q$ and every $f \in L_{1}^{\rm loc}(\mathcal{H}^{t}\lfloor_{S})$.

\subsection{Weights}

In this subsection we fix $d \in \mathbb{N}$, $d \geq 2$.
By \textit{a weight on $\mathbb{R}^{d}$} we always mean a locally integrable nonnegative on $\mathbb{R}^{d}$ Borel function $\gamma$ which is positive $\mathcal{L}^{d}$-almost everywhere.

\textbf{Important notation.} Given a weight $\gamma$ on $\mathbb{R}^{d}$, for each cube $Q$ in $\mathbb{R}^{d}$, we set
\begin{equation}
\label{eqq.weight_notation}
\gamma_{Q}:=\int\limits_{Q}\gamma(y)\,dy, \quad \underline{\gamma}_{Q}:=\essinf\limits_{y \in Q}\gamma(y), \quad \overline{\gamma}_{Q}:=\fint\limits_{Q}\gamma(y)\,dy.
\end{equation}
Given $x \in \mathbb{R}^{d}$ and $r > 0$, we also put
\begin{equation}
\label{eqq.weight_notation'}
\gamma_{r}(x):=\gamma_{Q_{r}(x)}, \quad \underline{\gamma}_{r}(x):=\underline{\gamma}_{Q_{r}(x)}, \quad \overline{\gamma}_{r}(x):=\overline{\gamma}_{Q_{r}(x)}.
\end{equation}

The concept of weights satisfying the \textit{local Muckenhoupt} condition was introduced by V.~Rychkov in \cite{Ry}. We use only a particular case of his definition related
to the case $p=1$.
\begin{Def}
\label{Def.A_1_weight}
By $A_{1}^{\rm loc}(\mathbb{R}^{d})$ we mean the set of all weights $\gamma$ on $\mathbb{R}^{d}$ such that
\begin{equation}
\label{eqq.A_1_weight}
\operatorname{C}_{\gamma}:=\sup\limits_{(x,r) \in \mathbb{R}^{d} \times (0,1]} \frac{\overline{\gamma}_{r}(x)}{\underline{\gamma}_{r}(x)} < +\infty.
\end{equation}
\end{Def}
Now we show that the measure $\gamma \mathcal{L}^{d}$ have some sort of locally uniformly doubling properties.
\begin{Lm}
\label{Lm.doubling}
For each $\gamma \in A_{1}^{\rm loc}(\mathbb{R}^{d})$ the following properties hold:
\begin{itemize}

\item[\((\text{D1})\)] if cubes $Q,Q'$ in $\mathbb{R}^{d}$ are such that $l(Q) \le 2$,  $l(Q')=2^{-1}l(Q)$ and $Q' \subset Q$, then
\begin{equation}
\label{eqq.doubling_1}
\underline{\gamma}_{Q} \le \underline{\gamma}_{Q'} \le 2^{d}\operatorname{C}_{\gamma}\underline{\gamma}_{Q}, \quad \overline{\gamma}_{Q'} \le 2^{d}\overline{\gamma}_{Q} \le 2^{d}\operatorname{C}_{\gamma}\overline{\gamma}_{Q'};
\end{equation}

\item[\((\text{D2})\)] if cubes $Q_{1}$, $Q_{2}$ have nonempty intersection and $l(Q_{1})=l(Q_{2}) \le 1$, then
\begin{equation}
\label{eqq.doubling_2}
\underline{\gamma}_{Q_{1}} \le 2^{d}\operatorname{C}_{\gamma}\underline{\gamma}_{Q_{2}}, \quad \overline{\gamma}_{Q_{1}} \le 2^{d}(\operatorname{C}_{\gamma})^{2}\overline{\gamma}_{Q_{2}};
\end{equation}

\item[\((\text{D3})\)] for each $c > 1$,
\begin{equation}
\label{eqq.A_1_weight'}
\operatorname{C}_{\gamma,c}:=\sup\limits_{(x,r) \in \mathbb{R}^{d} \times (0,c]} \frac{\overline{\gamma}_{r}(x)}{\underline{\gamma}_{r}(x)} < +\infty.
\end{equation}

\item[\((\text{D4})\)] if $E_{\gamma} \subset \mathbb{R}^{d}$ is the set of Lebesgue points of $\gamma$, then
\begin{equation}
\label{eqq.maximal_weigh_estimate}
\int\limits_{Q_{1}(x)}\frac{\gamma(y)}{\|x-y\|^{d-1}}\,dy \le 2^{d}\operatorname{C}_{\gamma}\gamma(x) \qquad \hbox{for all} \quad x \in E_{\gamma}.
\end{equation}
\end{itemize}
\end{Lm}

\begin{proof}
Note that \eqref{eqq.doubling_1} follows directly from \eqref{eqq.A_1_weight}. Indeed, by direct computations we get
\begin{equation}
\label{eqq.doubling_weight}
\underline{\gamma}_{Q} \le \underline{\gamma}_{Q'} \le \overline{\gamma}_{Q'} \le 2^{d}\overline{\gamma}_{Q} \le 2^{d}\operatorname{C}_{\gamma}\underline{\gamma}_{Q}.
\end{equation}

If cubes $Q_{1}$ and $Q_{2}$ have the same side length $\le 1$ and common boundary points, then taking a cube $Q$ with $l(Q) \le 2$ containing both of them and applying
\eqref{eqq.A_1_weight}, \eqref{eqq.doubling_weight} we get
a chain of inequalities proving \eqref{eqq.doubling_2}
\begin{equation}
\label{eqq.doubling_weight'}
\overline{\gamma}_{Q_{1}} \le \operatorname{C}_{\gamma}\underline{\gamma}_{Q_{1}} \le 2^{d}(\operatorname{C}_{\gamma})^{2}\underline{\gamma}_{Q} \le 2^{d}(\operatorname{C}_{\gamma})^{2}\underline{\gamma}_{Q_{2}} \le 2^{d}(\operatorname{C}_{\gamma})^{2}\overline{\gamma}_{Q_{2}}.
\end{equation}

To prove \eqref{eqq.A_1_weight'} fix an arbitrary point $x \in \mathbb{R}^{d}$ and a parameter $r \in (1,c]$. Let $\mathcal{F}$ be a family of all cubes from $\mathcal{D}_{0}$ that have nonempty intersections
with $Q=Q_{r}(x)$. Clearly for every $Q',Q'' \in \mathcal{F}$ there is a noneoverlapping family $\{Q_{i}\}_{i=1}^{N} \subset \mathcal{F}$ such that: $Q'=Q_{1}$, $Q''=Q_{N}$,
$\partial Q_{i} \cap \partial Q_{i+1} \neq \emptyset$ for each $i \in \{1,...,N-1\}$, and $N \le 2(r+1)$. Hence, by \eqref{eqq.doubling_2}
we have $\underline{\gamma}_{Q'} \le (2^{d}\operatorname{C}_{\gamma})^{2(r+1)}\underline{\gamma}_{Q''}$.
As a result, since $\#\mathcal{F} \le 2^{d}(r+1)^{d}$, we deduce the required estimate
\begin{equation}
\notag
\begin{split}
&\overline{\gamma}_{Q} \le \sum_{Q' \in \mathcal{F}}\frac{1}{(2r)^{d}}\overline{\gamma}_{Q'} \le \operatorname{C}_{\gamma}\sum_{Q' \in \mathcal{F}}\frac{1}{(2r)^{d}}\underline{\gamma}_{Q'}\\
&\le 2^{d}\operatorname{C}_{\gamma}(2^{d}\operatorname{C}_{\gamma})^{2(r+1)}\min\{\underline{\gamma}_{Q'}:Q'\in\mathcal{F}\} \le
(2^{d}\operatorname{C}_{\gamma})^{2r+3}\underline{\gamma}_{Q}.
\end{split}
\end{equation}

To verify $(\text{D4})$, given $x \in E_{\gamma}$, we consider the layers $\widetilde{Q}_{k}:=Q_{2^{-k}}(x) \setminus Q_{2^{-k-1}}(x)$, $k \in \mathbb{N}_{0}$.
Direct computations give
\begin{equation}
\begin{split}
&\int\limits_{Q_{1}(x)}\frac{\gamma(y)}{\|x-y\|^{d-1}}\,dy \le \sum\limits_{k=0}^{\infty} 2^{(k+1)(d-1)}\int\limits_{\widetilde{Q}_{k}}\gamma(y)\,dy \le \sum\limits_{k=0}^{\infty}\frac{2^{d}}{2^{k+1}}\fint\limits_{Q_{2^{-k}}(x)}\gamma(y)\,dy\\
&\le 2^{d}\operatorname{C}_{\gamma}\sum\limits_{k=0}^{\infty}2^{-k}\essinf\limits_{y \in Q_{2^{-k}}(x)}\gamma(y) \le 2^{d}\operatorname{C}_{\gamma}\gamma(x).
\end{split}
\end{equation}

The proof is complete.
\end{proof}


\subsection{Weighted Sobolev spaces}
Given an open set $\Omega \subset \mathbb{R}^{d}$, by $W^{1, \rm loc}_{1}(\Omega)$ we denote the linear space of (equivalence classes of) locally integrable on $\Omega$ functions
whose first order distributional partial derivatives are locally integrable (with respect to $\mathcal{L}^{d}$) on $\Omega$.
Given $\gamma \in A_{1}^{\rm loc}(\mathbb{R}^{d})$, we define the \textit{weighted Sobolev space} $W^{1}_{1}(\Omega,\gamma)$ as the linear space
$$
W^{1}_{1}(\Omega,\gamma):=\{F \in W^{1, \rm loc}_{1}(\Omega:\|F\|_{W^{1}_{1}(\Omega,\gamma)} < +\infty\},
$$
equipped with the norm
\begin{equation}
\label{eqq.Sobolev_norm}
\|F\|_{W^{1}_{1}(\Omega,\gamma)}:=\int\limits_{\Omega}\gamma(y)|F(y)|\,dy+\int\limits_{\Omega}\gamma(y)\|\nabla F(y)\|\,dy,
\end{equation}
where $\nabla F$ is the distributional gradient of $F$.

\begin{Remark}
\label{Rem.Sobolev_space_complete}
It is well-known and easy to show that $W^{1}_{1}(\Omega,\gamma)$ is a Banach space.
\end{Remark}

We recall (see Section 4.8 in \cite{Evans} for the details) that, for each $F \in W_{1}^{1, \rm loc}(\mathbb{R}^{d})$, there is a Borel set $E_{F} \subset \mathbb{R}^{d}$ such that $\mathcal{H}^{d-1}(E_{F})=0$ and
\begin{equation}
\label{eqq.sharp_representative}
\overline{F}(x):=\varlimsup\limits_{r \to +0}\fint\limits_{Q_{r}(x)}F(y)\,dy \in \mathbb{R} \quad \hbox{for all} \quad x \in \mathbb{R}^{d}\setminus E_{F},
\end{equation}
and, moreover, each $x \in \mathbb{R}^{d}\setminus E_{F}$ is a Lebesgue point of $\overline{F}$, i.e.
\begin{equation}
\label{eqq.Lebesgue_points}
\lim\limits_{r \to +0}\fint\limits_{Q_{r}(x)}|\overline{F}(x)-F(y)|\,dy = 0 \quad \hbox{for all} \quad x \in \mathbb{R}^{d}\setminus E_{F}.
\end{equation}
We refer $\overline{F}$ as the \textit{sharp representative} of $F$.

\begin{Def}
\label{Def.trace_space}
Given a Borel set $S \subset \mathbb{R}^{d}$ with $\mathcal{H}^{d-1}(S) > 0$, for each $F \in W^{1}_{1}(\mathbb{R}^{d},\gamma)$, the trace $F|_{S}$ of $F$ to $S$ is an $\mathcal{H}^{d-1}$-equivalence
class of the pointwise restriction of $\overline{F}$ to $S$, i.e.
\begin{equation}
F|_{S}:=\{f \in \mathfrak{B}(S): \mathcal{H}^{d-1}(\{x \in S:f(x) \neq \overline{F}(x)\})=0\}.
\end{equation}
The trace space of the weighted Sobolev space $W^{1}_{1}(\mathbb{R}^{d},\gamma)$ to $S$ is defined by the equality
\begin{equation}
W^{1}_{1}(\mathbb{R}^{d},\gamma):=\{F|_{S}:F \in W^{1}_{1}(\mathbb{R}^{d},\gamma)\}
\end{equation}
and it is equipped with the quotient-space norm, i.e., for each $f \in W^{1}_{1}(\mathbb{R}^{d},\gamma)|_{S}$,
\begin{equation}
\label{eqq.2.23}
\|f\|_{W^{1}_{1}(\mathbb{R}^{d},\gamma)|_{S}}:=\inf\{\|F\|_{W^{1}_{1}(\mathbb{R}^{d},\gamma)}: f=F|_{S}\}.
\end{equation}
\end{Def}

\begin{Remark}
\label{Rem.trace_space_complete}
Taking into account Remark \ref{Rem.Sobolev_space_complete} it is easy to show that the space $W^{1}_{1}(\mathbb{R}^{d},\gamma)|_{S}$ is a Banach space.
\end{Remark}

The following assertion will be important in proving the so-called ``direct trace theorem''.
\begin{Prop}
\label{Prop.Ziemer_estimate}
Let $S \in \mathcal{ADR}^{n}(\mathbb{R}^{n+1})$ for some $n \in \mathbb{N}$ and $c_{1},c_{2} > 0$.
Then, for each cube $Q$ in $\mathbb{R}^{n+1}$ with $l(Q) \le c_{1}$ and $\mathcal{H}^{n}(Q \cap S) \geq c_{2}(l(Q))^{n}$,
\begin{equation}
\label{eq.Ziemer_estimate'}
\fint\limits_{Q \cap S}\Bigl|F|_{S}(x)-\fint\limits_{Q}F(y)\,dy\Bigr|\,d\mathcal{H}^{n}(x) \lesssim l(Q)\fint\limits_{Q}\|\nabla F(y)\|\,dy;
\end{equation}
In particular,
\begin{equation}
\label{eq.Ziemer_estimate}
\mathcal{E}^{S}_{Q}[F|_{S}] \lesssim  l(Q)\fint\limits_{Q}\|\nabla F(y)\|\,dy.
\end{equation}
The corresponding constants in \eqref{eq.Ziemer_estimate'} and \eqref{eq.Ziemer_estimate} depend on $c_{1}$, $c_{2}$, $C_{S}^{1}$, $C_{S}^{2}$ and $n$ but do not depend neither $Q$, nor $F$.
\end{Prop}

\begin{proof}
Estimate \eqref{eq.Ziemer_estimate} follows from \eqref{eq.Ziemer_estimate'} by the triangle inequality in combination with Remark \ref{Rem.different_averagings}.
Estimate \eqref{eq.Ziemer_estimate} is an easy consequence of Theorem 5.12.7 in \cite{Zi}. Indeed, using notation adopted in \cite{Zi}
we set $\mu = \mathcal{H}^{n}\lfloor_{S \cap Q}$. For every $\varepsilon  \in (0,1)$ we take $\Omega_{\varepsilon}:=(1+\varepsilon)\operatorname{int}Q$.
Since $l(Q_{1}) \le c_{1}$ and $S \in \mathcal{ADR}^{n}(\mathbb{R}^{n+1})$, it is easy to conclude that $\mu$ satisfies assumptions of Theorem 5.12.7 in \cite{Zi}. 
This gives 
\begin{equation}
\label{eq.Ziemer_estimate'}
\fint\limits_{Q \cap S}\Bigl|F|_{S}(x)-\fint\limits_{Q}F(y)\,dy\Bigr|\,d\mathcal{H}^{n}(x) \le C (1+\varepsilon)l(Q)\fint\limits_{\Omega_{\varepsilon}}\|\nabla F(y)\|\,dy.
\end{equation}
Analysis of the proof of the theorem shows in combination with homothety arguments that $C > 0$ does not depend on $F$, $Q$ and $\varepsilon$. 
Hence, passing to the limit in \eqref{eq.Ziemer_estimate'} as $\varepsilon \to +0$, we deduce
\eqref{eq.Ziemer_estimate'}.
This completes the proof. 
\end{proof}


\section{Special families of cubes}

We recall Proposition \ref{Prop.porous_property}. Throughout the section we fix the following data:
\begin{itemize}

\item[\((\text{D.3.1})\)] a number $n \in \mathbb{N}$ and a weight $\gamma \in A_{1}^{\rm loc}(\mathbb{R}^{n+1})$;

\item[\((\text{D.3.2})\)] a set $S \in \mathcal{ADR}^{n}(\mathbb{R}^{n+1})$;

\item[\((\text{D.3.3})\)] a parameter $j \geq \max\{\underline{j}(S,2),5\}$.

\end{itemize}

\textbf{Important notation.} The following notation will be commonly used in the current and forthcoming paragraphs.
We put
\begin{equation}
\label{eqq.theta}
\theta:=\theta(j):=1+\frac{2}{2^{j}}+\frac{2}{2^{2j}}, \quad \underline{\theta}:=\underline{\theta}(j):=\theta(j)-\frac{2}{2^{4j}}.
\end{equation}
Now we set
\begin{equation}
\label{eqq.dyadic_S}
\widetilde{\mathcal{D}}_{k}(S):=\mathcal{D}_{jk}(S,\underline{\theta}) \quad \hbox{for each} \quad k \in \mathbb{Z}.
\end{equation}
Finally,
\begin{equation}
\label{eqq.dyadic_S_2}
\widetilde{\mathcal{D}}(S):=\bigcup_{k \in \mathbb{Z}}\widetilde{\mathcal{D}}_{k}(S), \quad \widetilde{\mathcal{D}}_{+}(S):=\bigcup_{k \in \mathbb{N}_{0}}\widetilde{\mathcal{D}}_{k}(S).
\end{equation}
We also use the abbreviation $\widetilde{\mathcal{D}}_{k}:=\widetilde{\mathcal{D}}_{k}(\mathbb{R}^{n+1})$, $\widetilde{\mathcal{D}}:=\widetilde{\mathcal{D}}(\mathbb{R}^{n+1})$  and $\widetilde{\mathcal{D}}_{+}:=\widetilde{\mathcal{D}}_{+}(\mathbb{R}^{n+1})$.
\begin{Remark}
\label{Rem.massive_part}
It is clear from Definition \ref{Def.Ahlfors_David} and \eqref{eqq.theta} that $\mathcal{H}^{n}(\theta Q \cap S) \approx (l(Q))^{n}$ for each $Q \in \widetilde{\mathcal{D}}_{+}(S)$.
The corresponding constants depend only on $n$, $j$ and $C^{S}_{1}$, $C^{S}_{2}$,
\end{Remark}

The following lemma is elementary. We present the proof for the completeness.
\begin{Lm}
\label{Lm.covering_multiplicity}
Let $\mathcal{F} \subset \widetilde{\mathcal{D}}$ be such that $2^{-k^{\ast}j} \le \frac{l(Q_{1})}{l(Q_{2})} \le 2^{k^{\ast}j}$ for some $k^{\ast} \in \mathbb{N}_{0}$
and every $Q_{1},Q_{2} \in \mathcal{F}$ with $cQ_{1} \cap cQ_{2} \neq \emptyset$.
Then,
\begin{equation}
\operatorname{MULT}(\{cQ:Q \in \mathcal{F}\}) \le (2k^{\ast}+1)([c]+2)^{n+1}.
\end{equation}
\end{Lm}

\begin{proof}
Fix an arbitrary $Q \in \mathcal{F} \cap \widetilde{\mathcal{D}}_{k}$ for some $k \in \mathbb{Z}$. If $Q' \in \mathcal{F}$ is such that $cQ' \cap cQ \neq \emptyset$, then
$Q' \in \widetilde{\mathcal{D}}_{k+j}$ for some $j \in \{-k^{\ast},...,k^{\ast}\}$. Hence, given $x \in cQ$, by Remark \ref{Rem.dyadic_lattice_multiplicity} we have
\begin{equation}
\notag
\#\{K \in \mathcal{F}:cK \ni x\}) \le \sum\limits_{j=k-k^{\ast}}^{k+k^{\ast}}\operatorname{MULT}\{cQ': Q' \in \widetilde{\mathcal{D}}_{j}(S)\} \le  (2k^{\ast}+1)([c]+2)^{n+1}.
\end{equation}
The proof is complete.
\end{proof}

\begin{Def}
\label{Def.neighboring_cubes}
Given $Q \in \widetilde{\mathcal{D}}$, a cube $Q'\in \widetilde{\mathcal{D}}$ is \textit{neighboring} for $Q$ if $Q \cap Q' \neq \emptyset$ and
\begin{equation}
\label{eqq.neighborin_cubes}
\frac{1}{2^{j}} \le \frac{l(Q)}{l(Q')} \le 2^{j}.
\end{equation}
By $\operatorname{n}(Q)$ we denote the family of all cubes $Q'$ neighboring for $Q$ and put $\operatorname{n}_{S}(Q):=\operatorname{n}(Q) \cap \widetilde{\mathcal{D}}(S)$.
\end{Def}

Besides neighboring cubes we will deal with cubes having ``less restrictive connections''.
\begin{Def}
\label{Def.relative_cubes}
Given $Q_{1} \in \widetilde{\mathcal{D}}$ we say that a dyadic cube $Q_{2}$ is relative for $Q_{1}$ if:
\begin{itemize}

\item[\((\text{Rel.1})\)] $Q_{2} \in \widetilde{\mathcal{D}}$ and $5Q_{2} \subset 5Q_{1}$;

\item[\((\text{Rel.2})\)] if $K \in \widetilde{\mathcal{D}}$ is such that $K \supset Q_{2}$ and $l(K) > l(Q_{2})$, then $5K$ is not contained in $5Q_{1}$.

\end{itemize}
Let $\mathcal{REL}(Q_{1})$ denote the family of all relatives for $Q_{1}$ cubes, $\mathcal{REL}_{S}(Q_{1}):=\mathcal{REL}(Q_{1})\cap \widetilde{\mathcal{D}}(S)$.
\end{Def}

The following elementary observation will be useful in proving important Proposition \ref{Prop.shadow_cubes_comparable} and Theorem \ref{Th.green_is_regular} below.
\begin{Lm}
\label{Lm.relative_cubes}
Let $Q_{1},Q_{2} \in \widetilde{\mathcal{D}}$ be such that $l(Q_{2}) < 2^{-j}l(Q_{1})$ and $\operatorname{dist}(Q_{1},Q_{2}) \le \frac{7}{4}l(Q_{1})$. Then, there
exists a cube $K \in \mathcal{REL}(Q_{1})$ such that $K \supset Q_{2}$ and $l(K)=2^{-j}l(Q_{1})$.
\end{Lm}

\begin{proof}
By \eqref{eqq.dyadic_S}, \eqref{eqq.dyadic_S_2} it follows that in fact $l(Q_{2}) \le 2^{-2j}l(Q_{1})$. Let $K \supset Q_{2}$ be the unique dyadic cube
with $l(K)=2^{-j}l(Q_{1})$. Since $\operatorname{dist}(Q_{1},K) \le \operatorname{dist}(Q_{1},Q_{2}) \le \frac{7}{4}l(Q_{1})$, we conclude that $5K \subset 5Q_{1}$ completing
the proof.
\end{proof}

Given $Q \in \widetilde{\mathcal{D}}(S)$ we define the \textit{box family} of $Q$ by
\begin{equation}
\label{eqq.box}
\mathcal{B}(Q):=\{Q' \in \widetilde{\mathcal{D}}(S): Q' \subset Q\}.
\end{equation}

The following concept is crucial.

\begin{Def}
\label{Def.covering_cube}
Given a family $\mathcal{Q} \subset \mathcal{D}$ and a cube $Q \in \mathcal{D}$, we say that a cube $\overline{Q}_{\mathcal{Q}} \in \mathcal{Q}$ is \textit{covering} for $Q$
with respect to $\mathcal{Q}$ if:
\begin{itemize}

\item[\((\text{C.1})\)] $Q \subset \overline{Q}_{\mathcal{Q}}$ and $l(\overline{Q}_{\mathcal{Q}}) > l(Q)$;

\item[\((\text{C.2})\)] if $Q' \in \mathcal{Q}$ is such that $Q \subset Q' \subset \overline{Q}_{\mathcal{Q}}$, then either $Q'=Q$ or $Q' = \overline{Q}_{\mathcal{Q}}$.

\end{itemize}

\end{Def}

Having at our disposal Definition \ref{Def.covering_cube} we are ready to
define a special family of dyadic cubes which will play a role of a skeleton for the extension operator.

\begin{Def}
\label{Def.regular_family}
We say that a family $\mathcal{Q} \subset \widetilde{\mathcal{D}}_{+}(S)$ is regular for $S$ provided that:
\begin{itemize}

\item[\((\text{R.1})\)] $\mathcal{D}_{0}(S) \subset \mathcal{Q}$;

\item[\((\text{R.2})\)] If $Q \in \mathcal{Q}$ and $Q' \in \mathcal{REL}_{S}(Q)$, then there is a cube $K \in \mathcal{Q}\setminus\{\overline{Q}_{\mathcal{Q}}\}$ such that
$K \supset Q'$ and $l(K) \le 2^{j}l(Q)$;

\item[\((\text{R.3})\)] for each $Q \in \mathcal{Q}$ with $l(Q) < 1$ there is $\overline{Q}_{\mathcal{Q}}$.

\end{itemize}
The class of all regular for $S$ families will be denoted by $\mathfrak{R}(S)$.
\end{Def}

While the following result follows immediately from Definition \ref{Def.regular_family}, it will be important for the nonlinear algorithm in Section 6.

\begin{Prop}
\label{Prop.union_of_regular_families}
If $\mathcal{Q}_{1},\mathcal{Q}_{2} \in \mathfrak{R}(S)$, then $\mathcal{Q}=\mathcal{Q}_{1} \cup \mathcal{Q}_{2} \in \mathfrak{R}(S)$.
\end{Prop}

Now we construct elementary building blocks for a special partition of unity

\begin{Def}
\label{Def.shadow_iceberg}
Let $\mathcal{Q} \in \mathfrak{R}(S)$. For each $Q \in \mathcal{Q}$, we define the shadow of $Q$ by
\begin{equation}
\label{eqq.shadow}
\mathcal{SH}_{\mathcal{Q}}(Q):=\{Q' \in \mathcal{Q}: \overline{Q'}_{\mathcal{Q}} = Q\},
\end{equation}
and, furthermore, the iceberg of $Q$ with respect to $\mathcal{Q}$ by
\begin{equation}
\label{eqq.iceberg}
\mathcal{IC}_{\mathcal{Q}}(Q):=\mathcal{B}(Q) \setminus \bigcup\limits_{Q' \in \mathcal{SH}_{\mathcal{Q}}(Q)}\mathcal{B}(Q').
\end{equation}
\end{Def}

The following assertion is an immediate consequence of Definitions \ref{Def.covering_cube} and \ref{Def.shadow_iceberg}.

\begin{Lm}
\label{Lm.shadow_nonoverlapping}
For each $Q \in \mathcal{Q}$,
the family  $\mathcal{IC}_{\mathcal{Q}}(Q)$ is noneoverlapping.
\end{Lm}



The following assertion is an easy consequence of Definition \ref{Def.shadow_iceberg}.

\begin{Lm}
\label{Lm.disjoint_iceberg}
If $\mathcal{Q} \in \mathfrak{R}(S)$, then the family $\mathcal{ICF}(\mathcal{Q}):=\{\mathcal{IC}_{\mathcal{Q}}(Q):Q \in \mathcal{Q}\}$ has the following properties:
\begin{itemize}

\item[\((\text{IC.1})\)] $\bigcup\limits_{Q \in \mathcal{Q}}\mathcal{IC}_{\mathcal{Q}}(Q)=\widetilde{\mathcal{D}}_{+}(S)$;

\item[\((\text{IC.2})\)] the family $\mathcal{ICF}(\mathcal{Q})$ is disjoint.

\end{itemize}
\end{Lm}

Now we clarify some combinatorial structure of shadows of cubes.

\begin{Prop}
\label{Prop.shadow_cubes_comparable}
Let $\mathcal{Q} \in \mathfrak{R}(S)$, $Q \in \mathcal{Q}$. For every pair of different cubes  $Q_{1},Q_{2} \in \mathcal{SH}_{\mathcal{Q}}(Q)$ the following properties hold:

\begin{itemize}

\item[\((\text{SH.1})\)] if $l(Q_{2}) < 2^{-j}l(Q_{1})$, then $\operatorname{dist}(Q_{1},Q_{2}) > \frac{7}{4}l(Q_{1})$;

\item[\((\text{SH.2})\)] if $4Q_{1} \cap 4Q_{2} \neq \emptyset$, then $2^{-j} \le \frac{l(Q_{1})}{l(Q_{2})} \le 2^{j}$.

\end{itemize}

\end{Prop}

\begin{proof}

To prove $(\text{SH.1})$ we assume that $\operatorname{dist}(Q_{1},Q_{2}) \le \frac{7}{4}l(Q_{1})$. Since $Q_{2} \in \widetilde{\mathcal{D}}_{+}(S)$, by Lemma \ref{Lm.relative_cubes}
we get existence of a cube $Q'_{2} \in \mathcal{REL}_{S}(Q_{1})$ such that $Q'_{2} \supset Q_{2}$ and $l(Q'_{2}) \geq 2^{j}l(Q_{2})$.
By Definition  \ref{Def.regular_family}
this implies existence of $K \supset Q'_{2}$ such that $K \in \mathcal{Q}$ and $K$ is strictly contained in $Q$. Hence, the cube $Q$ is not covering for $Q_{2}$. As a result, $Q_{2} \notin \mathcal{SH}_{\mathcal{Q}}(Q_{1})$.
We get a contradiction and complete the proof.

To prove $(\text{SH.2})$ we assume the contrary. Suppose that $l(Q_{2}) < 2^{-j}l(Q_{2})$.
Using $(\text{SH.1})$ and taking into account ($\text{D.3.3}$) it is easy to see that $4Q_{1} \cap 4Q_{2} = \emptyset$. This gives a contradiction.

The proof is complete.
\end{proof}

As in easy consequence we derive the following important observation.

\begin{Prop}
\label{Prop.covering_multiplicity_shadow}
If $\mathcal{Q} \in \mathfrak{R}(S)$ and $Q \in \mathcal{Q}$, then
\begin{equation}
\notag
\operatorname{MULT}(\{4 Q': Q' \in \mathcal{SH}_{\mathcal{Q}}(Q)\}) \le 5^{n+2}.
\end{equation}
\end{Prop}

\begin{proof}
Fix an arbitrary $Q' \in \mathcal{SH}_{\mathcal{Q}}(Q)$. Clearly $Q \in \widetilde{\mathcal{D}}_{k}(S)$ for some $k \in \mathbb{N}_{0}$.
By Proposition \ref{Prop.shadow_cubes_comparable}, if $Q'' \cap Q' \neq \emptyset$ for some $Q'' \in \mathcal{SH}_{\mathcal{Q}}(Q)$, then
$Q'' \in \widetilde{\mathcal{D}}_{j}(S)$ for some $j \in \{k-1,k,k+1\}$. Hence, by Lemma \ref{Lm.covering_multiplicity} applied with $k^{\ast}=1$ and $c=4$
the proposition follows.
\end{proof}

For the further development we need to relate abstract regular families with the behaviour of a given weight. This motivate us to introduce the following notion.
\begin{Def}
\label{Def.admissible_for_weight}
Given a weight $\gamma \in A^{\rm loc}_{1}(\mathbb{R}^{n+1})$, a family of cubes $\mathcal{Q}$ is said to be admissible for $\gamma$ if there is $\operatorname{C}_{\operatorname{A}} > 0$ such that:
$\underline{\gamma}_{Q} \le \operatorname{C}_{\operatorname{A}}\underline{\gamma}_{K}$ for $K=\overline{Q}_{\mathcal{Q}}$.
\end{Def}


\section{Analysis of special functionals}
The aim of this paragraph is the careful analysis of the Besov-type functional and the Lebesgue-type functional that we briefly mentioned
in the introduction. Throughout the whole section we fix the following data:

\begin{itemize}

\item[\((\text{D.4.1})\)] a number $n \in \mathbb{N}$ and a weight $\gamma \in A_{1}^{\rm loc}(\mathbb{R}^{n+1})$;

\item[\((\text{D.4.2})\)] a set $S \in \mathcal{ADR}^{n}(\mathbb{R}^{n+1})$;

\item[\((\text{D.4.3})\)] a parameter $j \geq \max\{\underline{j}(S,2),5\}$;

\item[\((\text{D.4.4})\)] a parameter $q > (2^{n+1}\operatorname{C}_{\gamma})^{2j}$.

\end{itemize}

We recall notation \eqref{eqq.weight_notation} and make an elementary observation.
\begin{Lm}
\label{Lm.oscillation_on_neighboring_cubes}
Let $Q_{1},Q_{2} \in \widetilde{\mathcal{D}}_{+}(S)$ be such that $Q_{2} \in \operatorname{n}(Q_{1})$. Then
\begin{equation}
\label{eqq.oscillation_of_neighboring_cubes}
\underline{\gamma}_{Q_{2}} \le \frac{q}{2}\underline{\gamma}_{Q_{1}}.
\end{equation}
\end{Lm}

\begin{proof}
Without loss of generality we may assume that $2^{-j}l(Q_{1}) \le l(Q_{2}) \le l(Q_{1})$. Since $Q_{2} \cap Q_{1} \neq \emptyset$
we clearly have $Q_{1} \subset 2^{j+2}Q_{2}$. Application of Lemma \ref{Lm.doubling} gives \eqref{eqq.oscillation_of_neighboring_cubes}.
\end{proof}

\subsection{Green cubes} Given $k \in \mathbb{Z}$, we define an auxiliary family
$$
\widetilde{\mathcal{G}}^{k}_{\gamma}(q):=\{Q \in \widetilde{\mathcal{D}}_{+}(S):\underline{\gamma}_{5Q} \in [q^{k},q^{k+1})\}.
$$
Now let $\mathcal{G}^{k}_{\gamma}(q) \subset \widetilde{\mathcal{G}}^{k}_{\gamma}(q)$ denote the corresponding \textit{maximal} cubes.
In other words, $Q$ belongs to $\mathcal{G}^{k}$ if and only if it belongs to $\widetilde{\mathcal{G}}^{k}$
and there are no other cubes from $\widetilde{\mathcal{G}}^{k}$ containing it. Finally, we set
$\mathcal{G}_{\gamma}(q):=\bigcup_{k \in \mathbb{Z}}\mathcal{G}^{k}_{\gamma}(q)$. In what follows we refer to the family $\mathcal{G}_{\gamma}(q)$ as the
family of \textit{green cubes}, and paint each cube from $\mathcal{G}_{\gamma}(q)$ \textit{green}.
The deep sense of this family will be explained in Section 6 below.

Since $\gamma$ and $q$ are assumed to be fixed in this section, we omit them from the corresponding notation, i.e., we set
$\widetilde{\mathcal{G}}^{k}:=\widetilde{\mathcal{G}}^{k}_{\gamma}(q)$, $\mathcal{G}^{k}:=\mathcal{G}^{k}_{\gamma}(q)$, and $\mathcal{G}:=\mathcal{G}_{\gamma}(q)$.
\begin{Lm}
\label{Lm.green_noneoverlapping}
For each $k \in \mathbb{Z}$, the following holds:

\begin{itemize}

\item[\((1)\)] the family $\mathcal{G}^{k}$ is noneoverlapping;

\item[\((2)\)] $\mathcal{G}^{k} \cap \mathcal{G}^{k+1} = \emptyset$;

\item[\((3)\)] $\underline{\gamma}_{25Q} < q^{k+1}$ for each $Q \in \mathcal{G}^{k}$ with $l(Q) < 1$.

\end{itemize}
\end{Lm}

\begin{proof}
Fix to the end of the proof an arbitrary $k \in \mathbb{Z}$.

To prove (1) assume the contrary. Since $\mathcal{G}^{k} \subset \widetilde{\mathcal{D}}(S)$ this implies existence of two different
cubes $Q_{1},Q_{2} \in \mathcal{G}^{k}$ such that one of them is strictly contained in the other one. This gives a contradiction with the maximality
of cubes in the family $\mathcal{G}^{k}$.

Property (2) follows from the maximality of the corresponding cubes.

To prove (3)  assume on the contrary that $\underline{\gamma}_{25Q} \geq q^{k+1}$ for some $Q \in \mathcal{G}^{k}$ with $l(Q) < 1$. Clearly, $Q \in \widetilde{\mathcal{D}}_{l}(S)$
for some $l \in \mathbb{N}$. By $(\text{D.4.4})$ and Lemma \ref{Lm.doubling} this leads to the inequality $\underline{\gamma}_{5K} > q^{k}$ for the unique cube $K \in \widetilde{\mathcal{D}}_{l-1}(S)$
containing $Q$. This gives a contradiction with the maximality of $Q$.

The proof is complete.
\end{proof}

Now we show that the family of green cubes possesses nice properties. 
\begin{Th}
\label{Th.green_is_regular}
The family $\mathcal{G}$ is regular for $S$ and admissible for $\gamma$.
\end{Th}

\begin{proof}
Given $Q \in \widetilde{\mathcal{D}}_{0}(S)=\mathcal{D}_{0}(S)$,  $\underline{\gamma}_{5Q} \in [q^{k},q^{k+1})$ for some $k \in \mathbb{N}$.
This verifies $(\text{R.1})$.

To verify $(\text{R.2})$ we take $Q \in \mathcal{G}^{k}$, $k \in \mathbb{Z}$ and $Q' \in \mathcal{REL}_{S}(Q)$. Hence, $\gamma_{5Q'} \geq q^{k}$.
Consequently, $Q' \in \widetilde{\mathcal{G}}^{k}$. This implies existence of $K \in \mathcal{G}^{k}$ such that $K \supset Q'$.
If we assume that $l(K) > 2^{j}l(Q)$, then by Lemma \ref{Lm.relative_cubes} we deduce existence of a cube $\widetilde{Q} \in \mathcal{REL}_{S}(K)$ that
strictly contains $Q$. Hence, $\underline{\gamma}_{5\widetilde{Q}} \geq \underline{\gamma}_{5K} \geq q^{k}$. This contradicts to the maximality of $Q$.
It remains to note that $K \neq \overline{Q}_{\mathcal{G}}$ due to Lemma \ref{Lm.green_noneoverlapping}.

To verify $(\text{R.3})$ fix $Q \in \mathcal{G}^{k}$ with $l(Q) < 1$. We claim that there is $K \supset Q$ such that $K \in \widetilde{\mathcal{D}}_{+}(S)$
and $\underline{\gamma}_{5K} < q^{k}$. Indeed, otherwise we take a unique cube $Q' \in \widetilde{\mathcal{D}}_{0}(S)$ containing $Q$ with the
property $\underline{\gamma}_{5Q'} \geq q^{k}$ contradicting the maximality of $Q'$. Hence, $K \in \widetilde{\mathcal{G}}^{k-1}$. The required property
now follows easily.

Finally, the family $\mathcal{G}$ is admissible for $\gamma$ by the construction.
\end{proof}

Now we show that $\mathcal{G}$ has good combinatorial properties.
\begin{Prop}
\label{Prop.green_multiplicity}
For each $k \in \mathbb{Z}$, $\operatorname{MULT}(\{4Q:Q \in \mathcal{G}^{k}\}) \le 5^{n+2}$.
\end{Prop}

\begin{proof}
Given $k \in \mathbb{Z}$, if $4Q_{1} \cap 4Q_{2} \neq \emptyset$ for some $Q_{1},Q_{2} \in \mathcal{G}^{k}$, then
\begin{equation}
\label{eqq.green_comparable}
2^{-j} \le \frac{l(Q_{1})}{l(Q_{2})} \le 2^{j}.
\end{equation}
Indeed, assume that $l(Q_{1}) > 2^{j}l(Q_{2})$. Then, $\operatorname{dist}(Q_{1},Q_{2}) \le \frac{7}{4}l(Q_{1})$.
By Lemma \ref{Lm.relative_cubes} this implies existence of a cube $K \in \mathcal{REL}_{S}(Q_{1})$ with $l(K)=2^{-j}l(Q_{1}) > l(Q_{2})$.
As a result, we have $\underline{\gamma}_{5K} \geq q^{k}$ and get a contradiction with the maximality of $Q_{2}$.
It remains to apply Lemma \ref{Lm.covering_multiplicity} with $c=4$ and $k^{\ast}=1$.
\end{proof}

\subsection{New Besov-type and Lebesgue-type functionals}
For each $k \in \mathbb{Z}$, we set
\begin{equation}
\label{eqq.set_k}
S^{k}:=\{x \in S: \varlimsup\limits_{r \to +0}\underline{\gamma}_{r}(x) > q^{k}\}.
\end{equation}
Given $k \in \mathbb{Z}$ and $x \in S^{k}$, we introduce the \textit{$k$-th stopping time} by the equality
\begin{equation}
\label{eqq.stopping_time_k}
r_{k}(x):=\sup\{r \in (0,1]: \underline{\gamma}_{5r}(x) \geq q^{k}\}.
\end{equation}
Given $k \in \mathbb{Z}$, for each $f \in L_{1}^{\rm loc}(\mathcal{H}^{n}\lfloor_{S})$, we put
$$
\mathcal{E}_{k}[f](x):=\mathcal{E}^{S}_{Q_{r_{k}(x)}(x)}[f], \quad x \in S
$$
and define the \textit{Besov-type functional} $\mathcal{BN}_{\gamma,q}:L_{1}^{\rm loc}(\mathcal{H}^{n}\lfloor_{S}) \to [0,+\infty]$, by
\begin{equation}
\label{eqq.Besov_functional}
\mathcal{BN}_{\gamma,q}[f]:=\sum\limits_{k \in \mathbb{Z}}q^{k}\int\limits_{S^{k}}\mathcal{E}_{k}[f](x)\,d\mathcal{H}^{n}(x), \quad f \in L_{1}^{\rm loc}(\mathcal{H}^{n}\lfloor_{S}).
\end{equation}
We recall \eqref{eqq.weight_notation'} and introduce the \textit{Lebesgue-type functional} $\mathcal{LN}_{\gamma}:L_{1}^{\rm loc}(\mathcal{H}^{n}\lfloor_{S}) \to [0,+\infty]$ by
\begin{equation}
\label{eqq.Lebesgue_functional}
\mathcal{LN}_{\gamma}[f]:=\int\limits_{S}\gamma_{1}(y)|f(y)|\,d\mathcal{H}^{n}(y), \quad f \in L_{1}^{\rm loc}(\mathcal{H}^{n}\lfloor_{S}).
\end{equation}
Finally, the \textit{keystone functional} $\mathcal{N}_{\gamma,q}:L_{1}^{\rm loc}(\mathcal{H}^{n}\lfloor_{S}) \to [0,+\infty]$ is given by the equality
\begin{equation}
\label{eqq.Norm_functional}
\mathcal{N}_{\gamma,q}[f]:=\mathcal{LN}_{\gamma}[f]+\mathcal{BN}_{\gamma,q}[f], \quad f \in L_{1}^{\rm loc}(\mathcal{H}^{n}\lfloor_{S}).
\end{equation}

\subsection{Discretization of the Lebesgue-type functional}
It will be also useful to have at our disposal an appropriate ``dyadic discretization'' of $\mathcal{LN}_{\gamma}[f]$. More precisely, given $c > \underline{\theta}$, for each
$f \in L_{1}^{\rm loc}(\mathcal{H}^{n}\lfloor_{S})$, we set
\begin{equation}
\label{eqq.Lebesgue_functional'}
\widetilde{\mathcal{LN}}_{\gamma,c}[f]:=\sum\limits_{Q \in \widetilde{\mathcal{D}}_{0}(S)}\gamma_{Q}\int\limits_{c Q \cap S}|f(y)|\,d\mathcal{H}^{n}(y).
\end{equation}

In the next assertion we show that the above discretization is admissible in a sense.
\begin{Prop}
\label{Prop.discretization_of_Lebesgue}
Given $c > \underline{\theta}$, for each $f \in L_{1}^{\rm loc}(\mathcal{H}^{n}\lfloor_{S})$,
\begin{equation}
\label{eqq.different_lebesgue_type}
\mathcal{LN}_{\gamma}[f] \approx \widetilde{\mathcal{LN}}_{\gamma,c}[f],
\end{equation}
where the corresponding equivalence constants do net depend on $f$.
\end{Prop}

\begin{proof}
Since $c > 1$, given $Q \in \mathcal{D}_{0}$ and $y \in cQ$, we have $cQ \subset Q_{2c}(y)$ and $Q_{1}(y) \subset (c+2)Q$. Hence, application of Lemma \ref{Lm.doubling}
gives
\begin{equation}
\label{eqq.gamma_equivalence}
\gamma_{Q} \approx \gamma_{1}(y) \quad \hbox{for all} \quad y \in cQ,
\end{equation}
where the corresponding equivalence constants depend only on $n$, $\operatorname{C}_{\gamma}$ and $c$.

Since $c > \underline{\theta}$, by \eqref{eqq.dyadic_S} we have $S \subset \bigcup_{Q \in \widetilde{\mathcal{D}}_{0}(S)}cQ$. Hence, using \eqref{eqq.gamma_equivalence} and taking
into account Remark \ref{Rem.dyadic_lattice_multiplicity} we obtain
\begin{equation}
\notag
\widetilde{\mathcal{LN}}_{\gamma,c}[f] \lesssim \sum\limits_{Q \in \widetilde{\mathcal{D}}_{0}(S)}\int\limits_{c Q \cap S}\gamma_{1}(y)|f(y)|\,d\mathcal{H}^{n}(y)
\lesssim \mathcal{LN}_{\gamma}[f].
\end{equation}

To prove the opposite estimate it is sufficient to use \eqref{eqq.gamma_equivalence}. This gives
\begin{equation}
\notag
\mathcal{LN}_{\gamma}[f] \lesssim \sum\limits_{Q \in \widetilde{\mathcal{D}}_{0}(S)}\int\limits_{c Q \cap S}\gamma_{1}(y)|f(y)|\,d\mathcal{H}^{n}(y)
\lesssim \sum\limits_{Q \in \widetilde{\mathcal{D}}_{0}(S)}\gamma_{Q}\int\limits_{c Q \cap S}|f(y)|\,d\mathcal{H}^{n}(y).
\end{equation}

The proof is complete.
\end{proof}

\subsection{Discretization of the Besov-type functional}
We introduce a discrete version of the functional $\mathcal{BN}_{\gamma,q}$. More, precisely, for each $f \in L_{1}^{\rm loc}(\mathcal{H}^{n}\lfloor_{S})$,
\begin{equation}
\label{eqq.Besov_functional'}
\widetilde{\mathcal{BN}}_{\gamma,q}[f]:=\sum\limits_{k \in \mathbb{Z}}\sum\limits_{Q \in \mathcal{G}^{k}}\frac{\gamma_{Q}}{l(Q)}\mathcal{E}^{S}_{4 Q}[f].
\end{equation}
We show that in fact $\mathcal{BN}_{\gamma,q}$ and $\widetilde{\mathcal{BN}}_{\gamma,q}$ are equivalent in appropriate sense. We recall important notation at the end of Section 2.2.

\begin{Prop}
\label{Prop.discretization_of_Besov_1}
There is a constant $C > 0$ such that
\begin{equation}
\label{eqq.discretization_of_Besov_1}
\mathcal{BN}_{\gamma,q}[f] \le C \widetilde{\mathcal{BN}}_{\gamma,q}[f]  \quad \hbox{for all} \quad f \in L_{1}^{\rm loc}(\mathcal{H}^{n}\lfloor_{S}).
\end{equation}
\end{Prop}

\begin{proof}
We split the proof into several steps.

\textit{Step 1.} Given $k \in \mathbb{Z}$ and $x \in S^{k}$, we define the family
\begin{equation}
\notag
\widetilde{\mathcal{F}}^{k}(x):=\{Q \in \widetilde{\mathcal{D}}_{+}(S): \underline{\theta}Q \cap Q_{r_{k}(x)}(x) \cap S \neq \emptyset \hbox{ and }\underline{\gamma}_{5Q} \geq q^{k-1}\}.
\end{equation}
The crucial observation is that $\widetilde{\mathcal{F}}^{k}(x) \subset \widetilde{\mathcal{G}}^{k-1}$. Let $\mathcal{F}^{k}(x)$ be the subfamily of $\widetilde{\mathcal{F}}_{k}(x)$
composed of the corresponding maximal cubes. Clearly, $\mathcal{F}^{k}(x) \subset \mathcal{G}^{k-1}$.

Furthermore, we claim that
\begin{equation}
\label{eqq.4.12'}
Q_{\frac{r_{k}(x)}{2}}(x) \subset 4 Q  \quad \hbox{for every} \quad Q \in \mathcal{F}^{k}(x).
\end{equation}
To prove \eqref{eqq.4.12'}, it is sufficient to note that $l(Q) \geq r_{k}(x)$ for all $Q \in \mathcal{F}^{k}(x)$. To prove this inequality assume on the contrary that
$l(K) < r_{k}(x)$ for some $K \in \mathcal{F}^{k}(x)$. Note that $\underline{\gamma}_{5K} < q^{k}$ by Lemma \ref{Lm.green_noneoverlapping}. On the other hand,
$5K \subset Q_{5r_{k}(x)}(x)$ and $\underline{\gamma}_{5K} \geq q^{k}$. This contradiction prove the required inequality.

\textit{Step 2.} Given $x \in S^{k}$, we
put $\Omega = \bigcup_{Q \in \mathcal{F}^{k}(x)}4Q$ and note that $Q_{r_{k}(x)} \subset \Omega$.
Combine \eqref{eqq.4.12'} and Lemma \ref{Lm.subadditivity_best_approximation} applied with
$\mathfrak{m}=\mathcal{H}^{n}\lfloor_{S}$ (we take into account that $S \in \mathcal{ADR}^{n}(\mathbb{R}^{n+1})$). As a result, we deduce
\begin{equation}
\label{eqq.4.13'}
\mathcal{E}_{k}^{S}[f](x) \lesssim \mathcal{E}^{S}_{\Omega}[f] \lesssim \sum\limits_{Q \in \mathcal{F}^{k}(x)}\mathcal{E}^{S}_{4 Q}[f].
\end{equation}

\textit{Step 3.} Given $k \in \mathbb{Z}$ and $x_{1},x_{2} \in S^{k}$, for each $Q_{1} \in \mathcal{F}^{k}(x_{1})$ and $Q_{2} \in \mathcal{F}^{k}(x_{2})$, either $Q_{1}$ and $Q_{2}$ have disjoint interiors or
$Q_{1}=Q_{2}$. Indeed, otherwise since all the cubes under consideration are dyadic, there are $Q_{1} \in \mathcal{F}^{k}(x_{1})$ and $Q_{2} \in \mathcal{F}^{k}(x_{2})$
such that some of them is strictly contained into the other one. Assume for example that $Q_{1} \subset Q_{2}$ and $l(Q_{2}) > l(Q_{1})$. Obviously, $Q_{2} \in \widetilde{\mathcal{F}}^{k}(x_{1})$
contradicting the maximality of $Q_{1}$.
For each $k \in \mathbb{Z}$, we put
$$
\mathcal{F}^{k}:=\bigcup_{x \in S^{k}}\mathcal{F}^{k}(x).
$$

\textit{Step 4.}
Given $k \in \mathbb{Z}$ and  $Q \in \mathcal{F}^{k}$, let $\Pi_{S}(Q):=\{x \in S^{k}:Q_{r_{k}(x)}(x) \cap \underline{\theta} Q \neq \emptyset\}$. Clearly,
$x \in \Pi_{S}(Q)$ if and only if $Q \in \mathcal{F}^{k}(x)$. Furthermore, by \eqref{eqq.4.12'} we have $\Pi_{S}(Q) \subset 4Q$.
As a result, taking into account \eqref{eqq.4.13'} we arrive at
\begin{equation}
\notag
\begin{split}
&q^{k}\int\limits_{S^{k}}\mathcal{E}_{k}^{S}[f](x)\,d\mathcal{H}^{n}(x) \lesssim q^{k}\int\limits_{S^{k}}\Bigl(\sum\limits_{Q \in \mathcal{F}_{k}(x)}\mathcal{E}^{S}_{4Q}\Bigr)\,d\mathcal{H}^{n}(x)\\
&\lesssim \sum\limits_{Q \in \mathcal{F}_{k}}\underline{\gamma}_{Q}\Bigl(\int\limits_{\Pi_{S}(Q)}\,d\mathcal{H}^{n}(x)\Bigr)\mathcal{E}^{S}_{4Q}
\lesssim \widetilde{\mathcal{BN}}_{\gamma,q}[f].
\end{split}
\end{equation}

The proof is complete.
\end{proof}

\begin{Prop}
\label{Prop.discretization_of_Lebesgue_2}
There is a constant $C > 0$ such that, for each $f \in L_{1}^{\rm loc}(\mathcal{H}^{n}\lfloor_{S})$,
\begin{equation}
\label{eqq.415}
\widetilde{\mathcal{BN}}_{\gamma,q}[f] \le C \mathcal{N}_{\gamma,q}[f].
\end{equation}
\end{Prop}

\begin{proof} Fix $k \in \mathbb{Z}$ and $Q \in \mathcal{G}^{k}$. Since $\underline{\gamma}_{5Q} \geq q^{k}$, by \eqref{eqq.set_k} we clearly have
\begin{equation}
\label{eqq.forgotten_inclusion}
4Q \cap S \subset S^{k-1}.
\end{equation}


If $Q \in \widetilde{\mathcal{D}}_{l}(S)$ for some $l \in \mathbb{N}$, then $l(Q) \le 2^{-j}$.
We claim that in this case
\begin{equation}
\label{eqq.4.17}
r_{k-1}(x) \geq 4l(Q) \quad \hbox{for every} \quad x \in 4Q\cap S.
\end{equation}
Indeed, otherwise if we take $r=\frac{r_{k-1}(x)}{2^{7}}$, then $\underline{\gamma}_{5r}(x) < q^{k}$.
To prove this inequality assume on the contrary that $\underline{\gamma}_{5r}(x) \geq q^{k}$. We take into account that $j \geq 5$ and use $(\text{D.4.4})$.
Hence, $\underline{\gamma}_{2^{8}r}(x) \geq q^{k}(2^{n+1}\operatorname{C}_{\gamma})^{-6} > q^{k-1}$. Since $2^{8}r > r_{k-1}(x)$ we get a contradiction with the maximality of $r_{k-1}(x)$.
On the other hand, $Q_{5r}(x) \subset 5Q$ and consequently $\underline{\gamma}_{5r}(x) \geq q^{k}$.
This gives a contradiction.

By \eqref{eqq.4.17} and Lemma \ref{Lm.subadditivity_best_approximation} it is easy to see that for each $Q \in \mathcal{G}^{k} \setminus \mathcal{D}_{0}(S)$
\begin{equation}
\label{eqq.4.18}
\mathcal{E}^{S}_{4 Q}[f] \lesssim \inf\limits_{x \in S \cap 4Q}\mathcal{E}^{S}_{k-1}[f](x).
\end{equation}
We recall definition of the family $\mathcal{G}^{k}$, Definition \ref{Def.Ahlfors_David} and take into account \eqref{eqq.weight_notation}, \eqref{eqq.weight_notation'}. This leads to
\begin{equation}
\notag
\sum\limits_{Q \in \mathcal{G}^{k} \setminus \mathcal{D}_{0}(S)}\frac{\gamma_{Q}}{l(Q)}\mathcal{E}^{S}_{4 Q}[f] \lesssim q^{k-1}\sum\limits_{Q \in \mathcal{G}^{k}\setminus \mathcal{D}_{0}(S)}\mathcal{H}^{n}(4Q \cap S)\inf\limits_{x \in S \cap 4Q}\mathcal{E}^{S}_{k-1}[f](x).
\end{equation}
Hence, using \eqref{eqq.forgotten_inclusion} and taking into account Proposition \ref{Prop.green_multiplicity} we deduce
\begin{equation}
\label{eqq.418}
\sum\limits_{k \in \mathbb{Z}}\sum\limits_{Q \in \mathcal{G}^{k}\setminus \mathcal{D}_{0}(S)}\frac{\gamma_{Q}}{l(Q)}\mathcal{E}^{S}_{4 Q}[f] \lesssim \sum\limits_{k \in \mathbb{Z}}q^{k-1}
\sum\limits_{Q \in \mathcal{G}^{k}}\int\limits_{4Q \cap S}\mathcal{E}^{S}_{k-1}[f](x)\,d\mathcal{H}^{n}(x) \lesssim \mathcal{BN}_{\gamma,q}[f].
\end{equation}

Finally, if $Q \in \mathcal{D}_{0}(S)$, then application of Remark \ref{Rem.different_averagings} leads to
\begin{equation}
\notag
\mathcal{E}^{S}_{4 Q}[f] \lesssim \fint\limits_{4Q \cap S}|f(y)|\,d\mathcal{H}^{n}(y).
\end{equation}
It remains to recall that $\mathcal{D}_{0}(S) \subset \mathcal{G}$, use item (2) of Lemma \ref{Lm.green_noneoverlapping} and Remark \ref{Rem.dyadic_lattice_multiplicity}.
We obtain
\begin{equation}
\label{eqq.419}
\sum\limits_{k \in \mathbb{Z}}\sum\limits_{Q \in \mathcal{G}^{k}\cap \mathcal{D}_{0}(S)}\frac{\gamma_{Q}}{l(Q)}\mathcal{E}^{S}_{4 Q}[f] \lesssim
\sum\limits_{Q \in \mathcal{D}_{0}(S)}\gamma_{Q}\int\limits_{4Q \cap S}|f(y)|\,d\mathcal{H}^{n}(y).
\end{equation}

Combining \eqref{eqq.418}, \eqref{eqq.419} and applying Proposition \ref{Prop.discretization_of_Lebesgue} we deduce \eqref{eqq.415}.

The proof is complete.
\end{proof}

\subsection{The keystone property}

It is the result below that allow to make a trick for constructions of nonlinear extension
operators. In fact, the root of this trick (of course in a more elementary manner) can be traced back to the pioneering work by Gagliardo \cite{Gal}.

\begin{Th}
\label{Th.nonlinear}
If $\mathcal{N}_{\gamma,q}[f] < +\infty$, then
\begin{equation}
\label{eqq.nonlinear}
\sum\limits_{Q \in \widetilde{\mathcal{D}}_{k}(S)}\underline{\gamma}_{Q}\mathcal{E}^{S}_{4Q}[f] \to +0, \qquad k \to +\infty.
\end{equation}
\end{Th}

\begin{proof}
We fix $\varepsilon > 0$ and split the proof into several steps.

\textit{Step 1.} We fix $k_{\varepsilon} \in \mathbb{N}$ and then $N_{\varepsilon} \in \mathbb{N}$ so large that (we set $Q^{c}_{N_{\varepsilon}}(0):=\mathbb{R}^{n+1}\setminus Q_{N_{\varepsilon}}(0)$)
\begin{equation}
\label{eqq.4.20}
\sum\limits_{|l| > k_{\varepsilon}}\sum\limits_{Q \in \mathcal{G}^{l}}\frac{\gamma_{Q}}{l(Q)}\mathcal{E}^{S}_{4 Q}[f] + \sum\limits_{l=-k_{\varepsilon}}^{k_{\varepsilon}}\sum\limits_{\substack{Q \in \mathcal{G}^{l} \\ Q \cap Q^{c}_{N_{\varepsilon}}(0) \neq \emptyset}}\frac{\gamma_{Q}}{l(Q)}\mathcal{E}^{S}_{4Q}[f]
< \varepsilon.
\end{equation}

\textit{Step 2.} We claim that, for each $f \in L_{1}^{\rm loc}(\mathcal{H}^{n}\lfloor_{S})$,
\begin{equation}
\label{eqq.4.22}
\Sigma_{k}(\varepsilon):=\sum\limits_{\substack{Q \in \widetilde{\mathcal{D}}_{k}(S) \\ Q \subset Q_{N_{\varepsilon}+1}(0)}}\int\limits_{4Q \cap S}\bigl|f(y)-M^{S}_{4 Q}[f]\bigr|\,d\mathcal{H}^{n}(y) \to 0, \quad k \to \infty.
\end{equation}
To prove \eqref{eqq.4.22} we proceed as follows. Given $k \in \mathbb{N}_{0}$ and $Q \in \widetilde{\mathcal{D}}_{k}(S)$, we have $4Q \subset Q_{\frac{4}{2^{jk}}}(y)$ for every $y \in 4Q$. Consequently,
\begin{equation}
\notag
\bigl|f(y)-M^{S}_{4 Q}[f]\bigr| \lesssim \fint\limits_{Q_{\frac{4}{2^{jk}}}(y) \cap S}|f(z)-f(y)|\,d\mathcal{H}^{n}(z).
\end{equation}
Combining this observation with Remark \ref{Rem.dyadic_lattice_multiplicity} we obtain
\begin{equation}
\label{eqq.4.22'''}
\Sigma_{k}(\varepsilon) \lesssim
\int\limits_{Q_{N_{\varepsilon}+3}(0) \cap S}\fint\limits_{Q_{\frac{4}{2^{jk}}}(y) \cap S}|f(y)-f(z)|\,d\mathcal{H}^{n}(z)\,d\mathcal{H}^{n}(y).
\end{equation}
The right-hand side of \eqref{eqq.4.22'''} tends to zero as $k \to \infty$. This is clear for continuous functions. In the general
case one should use a standard trick based on approximations by continuous functions (one should use Proposition 3.3.49 from \cite{HKST} with $\mu = \mathcal{H}^{n}\lfloor_{S \cap Q_{R}(0)}$ with $R > N_{\varepsilon}+5$).

\textit{Step 3.} We recall Definition \ref{Def.covering_cube}. Given $k \in \mathbb{N}_{0}$, we define the family
\begin{equation}
\label{eqq.4.23}
\mathcal{F}_{k}(\varepsilon):=\{Q \in \widetilde{\mathcal{D}}_{k}(S): \overline{Q}_{\mathcal{G}} \in \mathcal{G}^{l} \\ \hbox{ with } l \in [-k_{\varepsilon},k_{\varepsilon}]\}, \quad  \mathcal{B}_{k}(\varepsilon):=\widetilde{\mathcal{D}}_{k}(S) \setminus \mathcal{F}_{k}(\varepsilon).
\end{equation}
Clearly, if $Q \in \widetilde{\mathcal{D}}_{k}(S)$, then $4Q \supset Q_{1}(x)$ for some $x \in S$. 
Hence, taking into account Definition \ref{Def.Ahlfors_David}, notation \eqref{eqq.weight_notation} and Remark \ref{Rem.different_averagings} by \eqref{eqq.4.22} we obtain
\begin{equation}
\begin{split}
\label{eqq.4.24}
&\sum\limits_{\substack{Q \in \mathcal{F}_{k}(\varepsilon) \\ Q \cap Q_{N_{\varepsilon}}(0)\neq \emptyset}}\frac{\gamma_{Q}}{l(Q)}\mathcal{E}^{S}_{4 Q}[f] \le \operatorname{C}_{\gamma}\sum\limits_{\substack{Q \in \mathcal{F}_{k}(\varepsilon) \\ Q \cap Q_{N_{\varepsilon}}(0)\neq \emptyset}}\underline{\gamma}_{Q}(l(Q))^{n}\mathcal{E}^{S}_{4 Q}[f]\\
&\le \frac{\operatorname{C}_{\gamma}}{C^{S}_{1}} \sum\limits_{\substack{Q \in \mathcal{F}_{k}(\varepsilon) \\ Q \subset Q_{N_{\varepsilon}+1}(0)}}\int\limits_{4Q \cap S}\bigl|f(y)-M^{S}_{4Q}[f]\bigr|\,d\mathcal{H}^{n}(y) \to 0, \quad k \to \infty.
\end{split}
\end{equation}

\textit{Step 4.} We set $\operatorname{c}(Q):=\operatorname{c}(4Q, \mathcal{H}^{n}\lfloor_{S},f)$ for brevity. By Remark \ref{Rem.the_best_constant}, for each cube $Q \in \mathcal{G}$,
\begin{equation}
\label{eqq.4.25}
\begin{split}
&\sum\limits_{Q' \in \mathcal{SH}_{\mathcal{G}}(Q)}\underline{\gamma}_{Q'}\int\limits_{4 Q' \cap S}|f(y)-\operatorname{c}(Q')|\,d\mathcal{H}^{n}(y)\\
&\lesssim \sum\limits_{Q' \in \mathcal{SH}_{\mathcal{G}}(Q)}\underline{\gamma}_{Q}\int\limits_{4 Q' \cap S}|f(y)-M^{S}_{4Q}[f]|\,d\mathcal{H}^{n}(y) \lesssim
\underline{\gamma}_{Q}\int\limits_{4 Q \cap S}|f(y)-M^{S}_{4Q}[f]|\,d\mathcal{H}^{n}(y).
\end{split}
\end{equation}
Hence, using Remark \ref{Rem.the_best_constant} and  \eqref{eqq.4.23} we deduce
\begin{equation}
\label{eqq.4.26}
\begin{split}
&\sum\limits_{\substack{Q' \in \mathcal{B}_{k}(\varepsilon)}}\frac{\gamma_{Q'}}{l(Q')}\mathcal{E}^{S}_{4 Q'}[f] \lesssim
\sum\limits_{\substack{Q' \in \mathcal{B}_{k}(\varepsilon) }}\underline{\gamma}_{Q'}\int\limits_{4 Q' \cap S}|f(y)-\operatorname{c}(Q')|\,d\mathcal{H}^{n}(y) \lesssim
\sum\limits_{|l| > k_{\varepsilon}}\sum\limits_{Q \in \mathcal{G}^{l}}\gamma_{Q}\mathcal{E}^{S}_{4Q}[f].
\end{split}
\end{equation}

\textit{Step 5.} It is clear that if $Q \cap Q^{c}_{N_{\varepsilon}}(0) \neq \emptyset$, then $\overline{Q}|_{\mathcal{Q}} \cap Q^{c}_{N_{\varepsilon}}(0) \neq \emptyset$. Hence, arguments
similar those used in \eqref{eqq.4.26} lead to
\begin{equation}
\label{eqq.4.27}
\sum\limits_{\substack{Q' \in \mathcal{F}_{k}(\varepsilon) \\ Q' \cap Q^{c}_{N_{\varepsilon}}(0)\neq \emptyset}}\gamma_{Q'}\mathcal{E}^{S}_{4 Q'}[f] \lesssim
\sum\limits_{l=-k_{\varepsilon}}^{k_{\varepsilon}}\sum\limits_{\substack{Q \in \mathcal{G}^{l} \\ Q \cap Q^{c}_{N_{\varepsilon}}(0) \neq \emptyset}}\gamma_{Q}\mathcal{E}^{S}_{4 Q}[f].
\end{equation}

\textit{Step 6.}
Collecting estimates \eqref{eqq.4.20},\eqref{eqq.4.24},\eqref{eqq.4.26},\eqref{eqq.4.27} we arrive at
\begin{equation}
\varlimsup\limits_{k \to \infty}\sum\limits_{Q \in \widetilde{\mathcal{D}}_{k}(S)}\underline{\gamma}_{Q}\mathcal{E}^{S}_{4Q}[f] \lesssim \varepsilon.
\end{equation}
Since $\varepsilon > 0$ can be chosen arbitrarily small, the theorem follows.

The proof is complete.
\end{proof}


\section{Extension operators}

The goal of this paragraph is to present a family of new extension operators associated to the regular families of cubes.
The driving ideas of that construction have some roots in the authors
previous investigations \cite{T1,T3}. However, those ideas are not applicable directly to the present context. Indeed, methods developed in \cite{T1}
were essential based on the linear structure of the set, whereas the methods in \cite{T3} can be successfully used only in the context of Sobolev $W_{p}^{1}(\mathbb{R}^{d})$-space with $p > 1$.
This forces us to make a mixture of the ideas from \cite{T1,T3} with new insights which allow to attack the case when $p=1$ and the set under consideration does not possess any smooth structure.
We strongly believe that the constructions  below will be useful in
many investigations related to the extensions of functions from nonsmooth sets.

Throughout this paragraph we fix the following data:

\begin{itemize}

\item[\((\text{D.5.1})\)] a number $n \in \mathbb{N}$ and a set $S \in \mathcal{ADR}^{n}(\mathbb{R}^{n+1})$;

\item[\((\text{D.5.2})\)] a parameter $j \geq \max\{\underline{j}(S,2),5\}$;

\item[\((\text{D.5.3})\)] a parameter $\epsilon = 2^{-(4+s)j}$ for some $s \in \mathbb{N}$;

\item[\((\text{D.5.4})\)] a family of cubes  $\mathcal{Q} \in \mathfrak{R}(S)$.

\end{itemize}

\subsection{Special partition of unity}
Let $\widetilde{\psi} \in C_{0}^{\infty}(\mathbb{R}^{n+1})$ be such that:

$(\widetilde{\psi}1)$ $\chi_{(1-\epsilon)Q_{0,0}}(x) \le \widetilde{\psi}(x) \le \chi_{(1+\epsilon)Q_{0,0}}(x)$ for all $x \in \mathbb{R}^{n+1}$ and furthermore $\widetilde{\psi}(x) \in (0,1)$ for all $x \in (1+\epsilon)\operatorname{int}Q_{0,0} \setminus (1-\epsilon)Q_{0,0}$;

$(\widetilde{\psi}2)$ $\sum_{m \in \mathbb{Z}^{n+1}}\widetilde{\psi}(x-m) = 1$ for all $x \in \mathbb{R}^{n+1}$.

It is easy to see from $(\text{D.5.3})$, $(\widetilde{\psi}1)$ and Remark \ref{Rem.dyadic_lattice_multiplicity} that
\begin{equation}
\label{eqq.multiplicity_psi}
\operatorname{MULT}(\{\operatorname{supp}\widetilde{\psi}_{Q}:Q \in \widetilde{\mathcal{D}}_{0}(S)\}) \le \operatorname{MULT}(\{(1+\epsilon)Q:Q \in \mathcal{D}_{0}\}) \le 3^{n+1}.
\end{equation}

\textbf{Important notation.} We recall \eqref{eqq.dyadic_S}. Given $k \in \mathbb{Z}$, if $Q \in \widetilde{\mathcal{D}}_{k}(S)$ is such that $Q=Q_{kj,m}$, then we put
$\widetilde{\psi}_{Q}(\cdot):=\widetilde{\psi}(\frac{\cdot-m}{2^{kj}})$ for brevity. In this case we have
\begin{equation}
\label{eqq.psi_gradient'}
\|\nabla \widetilde{\psi}_{Q}(x)\| \le \operatorname{C}_{\widetilde{\psi}}\sqrt{n}2^{kj} = \frac{\operatorname{C}_{\widetilde{\psi}}\sqrt{n}}{l(Q)} \quad \hbox{for all} \quad x \in \mathbb{R}^{n+1},
\end{equation}
where the constant $\operatorname{C}_{\widetilde{\psi}} > 0$ depends only on the concrete choice of the function $\widetilde{\psi}$ but does not depend on $k$ and $j$.

Note that the property $(\widetilde{\psi}2)$ is equivalent to
the validity, for each $k \in \mathbb{N}_{0}$, of the equality
\begin{equation}
\label{eqq.partition'}
\sum\limits_{Q \in \widetilde{\mathcal{D}}_{k}}\widetilde{\psi}_{Q}(x) = 1 \quad  \hbox{for all} \quad x  \in \mathbb{R}^{n+1}.
\end{equation}

Now we are going to introduce a special family of smooth functions that will play the role of the ``smooth Whitney-type decomposition''.
To this aim we introduce, for each $k \in \mathbb{N}_{0}$, an auxiliary function
\begin{equation}
\label{eqq.partition_g_k}
g_{k}(x):=\sum\limits_{Q \in \widetilde{\mathcal{D}}_{k}(S)}\widetilde{\psi}_{Q}(x), \quad x \in \mathbb{R}^{n+1}.
\end{equation}
Finally, given $k \in \mathbb{N}_{0}$ and $Q \in \widetilde{\mathcal{D}}_{k}(S)$, we set
\begin{equation}
\label{eqq.psi}
\psi_{Q}(x):=\widetilde{\psi}_{Q}(x)(1-g_{k+1}(x)), \quad x \in \mathbb{R}^{n+1}.
\end{equation}

\begin{Remark}
\label{Rem.massive_support}
It is clear from \eqref{eqq.partition_g_k}, \eqref{eqq.psi} and $(\widetilde{\psi}1)$ that if $Q \in \widetilde{\mathcal{D}}_{k}(S)$, then $\chi_{\Omega_{Q}} \le \psi_{Q}$, where
\begin{equation}
\notag
\Omega_{Q}:=(1-\epsilon)Q \setminus \bigcup\limits_{Q \in \widetilde{\mathcal{D}}_{k+1}(S)}(1+\epsilon)Q.
\end{equation}
Furthermore, by Proposition \ref{Prop.porous_property} and $(\text{D.5.2})$, we get
\begin{equation}
\notag
\mathcal{L}^{n+1}(\Omega_{Q}) \geq C\mathcal{L}^{n+1}(Q),
\end{equation}
where the constant $C > 0$ does not depend on $Q$.

In other words, for each $k \in \mathbb{N}_{0}$, the set of points, where each function $\psi_{Q} \equiv 1$, $Q \in \widetilde{\mathcal{D}}_{k}(S)$, is massive enough. It is this observation that will be indispensable tool in proving Theorem \ref{Th.ext_is_right_inverse}.
This informally justifies the way of introducing functions $\psi_{Q}$.
\end{Remark}

In the following assertion we collect basic properties of functions $\psi_{Q}$.

\begin{Prop}
\label{Prop.psi_properties}
The family $\{\psi_{Q}:Q\in \widetilde{\mathcal{D}}_{+}(S)\}$ has the following properties:

($\psi1$) If $k \in \mathbb{N}_{0}$ and $x \in \mathbb{R}^{n+1}$ are such that $\operatorname{dist}(x,S) \le 2^{-j(k+1)}(1+2\epsilon)$, then
$g_{k}(x)=1$;

($\psi2$) for each $k \in \mathbb{N}_{0}$,
\begin{equation}
\sum\limits_{l=k}^{\infty}\sum\limits_{Q \in \widetilde{\mathcal{D}}_{l}(S)}\psi_{Q}(x) = g_{k}(x), \quad x \in \mathbb{R}^{n+1};
\end{equation}

($\psi3$) if $\operatorname{supp}\psi_{Q} \cap \operatorname{supp}\psi_{Q'} \neq \emptyset$ for some $Q,Q' \in \widetilde{\mathcal{D}}_{+}(S)$, then $Q' \in \operatorname{n}_{S}(Q)$;

($\psi4$) for each $x \in \mathbb{R}^{n+1} \setminus S$ there exists $\delta_{x} > 0$ such that
\begin{equation}
\label{eqq.locallly_finite_covering}
\#\{Q \in \widetilde{\mathcal{D}}_{+}(S):\operatorname{supp}\psi_{Q} \cap Q_{\delta_{x}}(x) \neq \emptyset\} \le 3^{n+2};
\end{equation}

($\psi5$) There is a constant $\operatorname{C}_{\psi} > 0$ depending only on $\widetilde{\psi}$, $n$ and $j$ such that
\begin{equation}
\label{eqq.psi_gradient}
\|\nabla \psi_{Q}(x)\| \le \frac{\operatorname{C}_{\psi}}{l(Q)} \quad \hbox{for all} \quad x \in \mathbb{R}^{n+1}.
\end{equation}

\end{Prop}

\begin{proof}

To prove ($\psi1$) we fix an arbitrary $x \in \mathbb{R}^{n+1}$
such that $\operatorname{dist}(x,S) \le 2^{-j(k+1)}(1+2\epsilon)$. Hence, by ($\text{D.5.3}$) there is a point $x' \in S$ such that $\|x-x'\|_{\infty} \le 2^{-j(k+1)}(1+2^{-4j})$.
By \eqref{eqq.theta}, if $Q \in \widetilde{\mathcal{D}_{k}}$ is such that $x \in (1+\epsilon)Q$,
then $\underline{\theta}Q \ni x'$ and hence $Q \in \widetilde{\mathcal{D}_{k}}(S)$. This means that $Q \in \widetilde{\mathcal{D}_{k}}(S)$ for each $Q \in \widetilde{\mathcal{D}_{k}}$ with $\widetilde{\psi}_{Q}(x) > 0$.
This observation together with $(\widetilde{\psi}2)$ leads to the required equality
\begin{equation}
\notag
g_{k}(x)=\sum\limits_{Q \in \widetilde{\mathcal{D}}_{k}}\widetilde{\psi}_{Q}(x) = 1.
\end{equation}

To prove ($\psi2$) it is sufficient to note that ($\psi1$) just proved leads to
\begin{equation}
\notag
\sum\limits_{Q \in \widetilde{\mathcal{D}}_{k}(S)}\psi_{Q}(x) = g_{k}(x)-g_{k+1}(x), \quad x \in \mathbb{R}^{n+1}.
\end{equation}

By ($\psi1$) and \eqref{eqq.psi}, given $k \in \mathbb{N}_{0}$ and $Q \in \widetilde{\mathcal{D}}_{k}(S)$, we immediately have
\begin{equation}
\notag
\operatorname{dist}(\operatorname{supp}\psi_{Q},S) > 2^{-j(k+2)}(1+\epsilon).
\end{equation}
Hence, by ($\widetilde{\psi}1$) this implies
\begin{equation}
\label{eqq.locallly_finite_covering'}
\operatorname{dist}(\operatorname{supp}\psi_{Q},\operatorname{supp}\psi_{Q'}) > 0
\end{equation}
for each
$k' > k+1$ and $Q' \in \widetilde{\mathcal{D}}_{k'}(S)$.  Combining this fact with Definition \ref{Def.neighboring_cubes} we get ($\psi3$).

Now ($\psi4$) is an easy consequence of ($\psi3$) and \eqref{eqq.locallly_finite_covering'}.

Finally, ($\psi5$) follows from \eqref{eqq.psi_gradient'} and \eqref{eqq.psi}.

The proof is complete.
\end{proof}

Now we are ready to construct the \textit{keystone family} of smooth functions which form the skeleton for our extension operator.
We recall Definition \ref{Def.shadow_iceberg} and, for each $Q \in \mathcal{Q}$, we set
\begin{equation}
\label{eqq.Psi}
\Psi_{Q}(x):=\sum\limits_{Q' \in \mathcal{IC}_{\mathcal{Q}}(Q)}\psi_{Q'}(x), \quad x \in \mathbb{R}^{n+1}.
\end{equation}

\begin{Remark}
\label{Rem.loc_finite_2}
We claim that, for each $x \in \mathbb{R}^{n+1}$,
\begin{equation}
\label{eqq.Psi_multiplicity}
\#\{Q \in \widetilde{\mathcal{D}}_{+}(S):\operatorname{supp}\Psi_{Q} \cap Q_{\delta_{x}}(x) \neq \emptyset\} \le 3^{n+2},
\end{equation}
where $\delta_{x} > 0$ is the same as in ($\psi4$) of Proposition \ref{Prop.psi_properties}.
Indeed, if $\underline{x} \in \operatorname{supp}\Psi_{Q}$, then we necessarily have $\underline{x} \in \operatorname{supp}\psi_{Q'}$ for some $Q' \in \mathcal{IC}_{\mathcal{Q}}(Q)$.
Now \eqref{eqq.Psi_multiplicity} follows directly from \eqref{eqq.locallly_finite_covering}.
\end{Remark}

\subsection{Abstract extension operator}
We start recall \eqref{eqq.maximal_function} and introduce a keystone concept.
\begin{Def}
\label{Def.abstract_extension}
A map $\operatorname{Ext}_{\mathcal{Q}}:L_{1}^{\rm loc}(\mathcal{H}^{n}\lfloor_{S}) \to C^{\infty}(\mathbb{R}^{n+1} \setminus S)$
defined, for each $f \in L_{1}^{\rm loc}(\mathcal{H}^{n}\lfloor_{S})$, by the equality
\begin{equation}
\label{eqq.abstract_extension}
\operatorname{Ext}_{\mathcal{Q}}[f](x):=\sum\limits_{Q \in \mathcal{Q}}M^{S}_{\theta Q}[f]\Psi_{Q}(x), \quad x \in \mathbb{R}^{n+1}
\end{equation}
is said to be an extension map associated with $\mathcal{Q}$.
\end{Def}

\begin{Remark}
\label{Rem.extension_operator_well_defined}
The map $\operatorname{Ext}_{\mathcal{Q}}$ is well defined.
Indeed, using Remark \ref{Rem.loc_finite_2} and taking into account that
$\Psi_{Q} \in C^{\infty}(\mathbb{R}^{n+1})$, $\operatorname{supp}\Psi_{Q} \subset \mathbb{R}^{n+1} \setminus S$ for every $Q \in \mathcal{Q}$, we deduce that the summand in
the right hand side of \eqref{eqq.abstract_extension} defines a function from  $C^{\infty}(\mathbb{R}^{n+1} \setminus S)$.
Furthermore, the map $\operatorname{Ext}_{\mathcal{Q}}$ is linear.
\end{Remark}

The following concept will be crucial in analysis of the infinitesimal behavior of the functions $\operatorname{Ext}_{\mathcal{Q}}[f]$, $f \in L_{1}^{\rm loc}(\mathcal{H}^{n}\lfloor_{S})$.
We recall $(\text{IC.2})$ of Lemma \ref{Lm.disjoint_iceberg}.
\begin{Def}
\label{Def.contacting_family}
For each $Q_{1},Q_{2} \in \mathcal{Q}$, we define the contacting family
\begin{equation}
\label{eqq.contacting_family}
\mathcal{C}_{\mathcal{Q}}(Q_{1},Q_{2}):=\{Q' \in \mathcal{IC}_{\mathcal{Q}}(Q_{1}): \operatorname{n}_{S}(Q')\cap \mathcal{IC}_{\mathcal{Q}}(Q_{2})  \neq \emptyset\}.
\end{equation}
Furthermore, given $Q \in \mathcal{Q}$, we set
\begin{equation}
\label{eqq.contacting_family'}
\mathcal{C}_{\mathcal{Q}}(Q):=\bigcup\limits_{Q' \in \mathcal{Q}}\mathcal{C}_{\mathcal{Q}}(Q,Q').
\end{equation}
Finally, for each $Q \in \mathcal{Q}$, we define the family of cubes interacting with $Q$ with respect to $\mathcal{Q}$ by letting
\begin{equation}
\label{eqq.interacting_family}
\mathcal{I}_{\mathcal{Q}}(Q):=\{K \in \mathcal{Q}:\mathcal{C}_{\mathcal{Q}}(Q,K) \neq \emptyset\}.
\end{equation}
\end{Def}

\begin{Remark}
\label{Rem.interacting_cubes}
It is clear that $\mathcal{C}_{\mathcal{Q}}(Q_{1},Q_{2}) \neq \emptyset$ if and only if $\mathcal{C}_{\mathcal{Q}}(Q_{2},Q_{1}) \neq \emptyset$. Hence,
if $K \in \mathcal{I}_{\mathcal{Q}}(Q)$, then $Q \in \mathcal{I}_{\mathcal{Q}}(K)$.
\end{Remark}

Given $f \in L_{1}^{\rm loc}(\mathcal{H}^{n}\lfloor_{S})$ and $Q \in \mathcal{Q}$, we introduce the special set
\begin{equation}
\label{eqq.special_set}
E_{\mathcal{Q}}(f,Q):=\{x \in \operatorname{supp}\Psi_{Q}: \nabla\operatorname{Ext}_{\mathcal{Q}}[f](x) \neq 0\}.
\end{equation}
Furthermore, we put
\begin{equation}
\label{eqq.special_set_2}
E_{\mathcal{Q}}(f):=\bigcup\limits_{Q \in \mathcal{Q}}E_{\mathcal{Q}}(f,Q).
\end{equation}

The following assertion is a direct consequence of \eqref{eqq.interacting_family} and \eqref{eqq.special_set}.
\begin{Lm}
\label{Lm.nonzero_gradient}
For each $f\in L_{1}^{\rm loc}(\mathcal{H}^{n}\lfloor_{S})$ and $Q \in \mathcal{Q}$,
\begin{equation}
E_{\mathcal{Q}}(f,Q) \subset \bigcup\limits_{Q' \in \mathcal{C}_{\mathcal{Q}}(Q)}\operatorname{supp}\psi_{Q'}.
\end{equation}
\end{Lm}

Now we establish the keystone proposition. In fact it is the main technical ingredient for the careful estimates of the gradient of extensions in the integral norms.
\begin{Prop}
\label{Prop.main_geometric_ingredient}
There is a constant $C=C(n,j,\theta) > 0$ such that, for each $Q_{1},Q_{2} \in \mathcal{Q}$,
\begin{equation}
\label{eqq.5.14''}
\sum\limits_{Q' \in \mathcal{C}_{\mathcal{Q}}(Q_{1},Q_{2})}(l(Q'))^{n} \le C (\min\{l(Q_{1}),l(Q_{2})\})^{n}.
\end{equation}
\end{Prop}

\begin{proof}
We fix two different cubes $Q_{1},Q_{2} \in \mathcal{Q}$ such that $\mathcal{C}_{\mathcal{Q}}(Q_{1},Q_{2}) \neq \emptyset$. Without loss of generality we assume that $l(Q_{2}) \le l(Q_{1})$.
There are two cases to be considered. In the first case $\operatorname{int}Q_{2} \cap \operatorname{int}Q_{1} = \emptyset$ but $\partial Q_{1} \cap \partial Q_{2} \neq \emptyset$.
In the second case $Q_{2} \subset \operatorname{int}Q_{1}$.

\textit{Step 1.} Consider the first case and fix $Q' \in \mathcal{C}_{\mathcal{Q}}(Q_{1},Q_{2})$. This implies existence of a cube
\begin{equation}
\label{eqq.Q_two_prime}
Q'' \in \mathcal{IC}_{\mathcal{Q}}(Q_{2}) \quad \hbox{s.t.} \quad \operatorname{supp}\psi_{Q'} \cap \operatorname{supp}\psi_{Q''} \neq \emptyset \quad \hbox{and} \quad Q'' \in \operatorname{n}_{S}(Q').
\end{equation}
We recall (\text{D.5.3}) and take into account that $\operatorname{supp}\psi_{Q} \subset \operatorname{supp}\widetilde{\psi}_{Q}$ for any $Q \in \widetilde{\mathcal{D}}_{+}(S)$.
This gives $Q' \cap Q'' \neq \emptyset$. At the same time $\operatorname{int}Q' \subset \operatorname{int}Q_{1}$
and $\operatorname{int}Q'' \subset \operatorname{int}Q_{2}$. As a result, we have
\begin{equation}
\label{eqq.521''}
\partial Q' \cap \partial Q'' \cap \partial Q_{1} \cap \partial Q_{2} \neq \emptyset.
\end{equation}

The crucial observation is that
\begin{equation}
\label{eqq.Hausdorff_measure_estimate}
(l(Q'))^{n} \lesssim \mathcal{H}^{n}(\operatorname{supp}\psi_{Q'} \cap \partial Q_{2}) \quad \hbox{for each} \quad Q' \in \mathcal{C}_{\mathcal{Q}}(Q_{1},Q_{2}).
\end{equation}
To prove \eqref{eqq.Hausdorff_measure_estimate} note that, if $Q' \in \widetilde{\mathcal{D}}_{k}(S)$ for some $k \in \mathbb{N}_{0}$, then there is a cube
\begin{equation}
\label{eqq.Q_underline}
\underline{Q} \in \widetilde{\mathcal{D}}_{k+1} \setminus \widetilde{\mathcal{D}}_{k+1}(S) \quad \hbox{s.t.} \quad \underline{Q} \cap \operatorname{supp}\widetilde{\psi}_{Q'} \cap \partial Q_{2} \neq \emptyset.
\end{equation}
Indeed, assume the contrary. Hence, for each $K \in \widetilde{\mathcal{D}}_{k+1}$ having a nonempty intersection with $\operatorname{supp}\widetilde{\psi}_{Q'} \cap \partial Q_{2}$ we
would necessarily get $K \in \widetilde{\mathcal{D}}_{k+1}(S)$. Note that if $K \cap \operatorname{supp}\widetilde{\psi}_{Q'} \neq \emptyset$ for $K \in \widetilde{\mathcal{D}}_{k+1}$, then
by $(\text{D.5.3})$ and $(\widetilde{\psi}1)$ we get $K \cap Q' \neq \emptyset$. Hence, taking into account \eqref{eqq.521''} we obtain $K \cap \partial Q' \neq \emptyset$. 
As a result,  by \eqref{eqq.partition_g_k} (we also use notation \eqref{eqq.neighborhood})
\begin{equation}
\notag
g_{k+1}(x) = 1 \quad \hbox{for all} \quad x \in U_{\delta_{k}}(\partial Q' \cap \partial Q_{2}),
\end{equation}
where $\delta_{k}=2^{-j-1}l(Q')$. By \eqref{eqq.psi} this implies that
\begin{equation}
\label{eqq.5.19'}
\operatorname{dist}(\operatorname{supp}\psi_{Q'},\partial Q_{2}) \geq 2^{-j-1}l(Q').
\end{equation}
Combining \eqref{eqq.5.19'} and  $(\text{D.5.3})$ with $(\psi3)$ of Proposition \ref{Prop.psi_properties} we get $\operatorname{supp}\psi_{Q'} \cap \operatorname{supp}\psi_{K} = \emptyset$ for any cube $K \in \mathcal{IC}_{\mathcal{Q}}(Q_{2})$ which contradicts to
\eqref{eqq.Q_two_prime}. By the right-hand side of \eqref{eqq.Q_underline} it follows easily that in fact $\partial \underline{Q} \cap \partial Q' \cap \partial Q_{2} \neq \emptyset$. Hence,
\begin{equation}
Q_{\delta_{\epsilon}}(\underline{x}) \cap \partial Q_{2} \subset (1+\epsilon)\underline{Q} \cap (1+\epsilon)Q' \quad \hbox{for some} \quad \underline{x} \in \partial\underline{Q} \cap \partial Q' \cap \partial Q_{2},
\end{equation}
where $\delta_{\epsilon}:=\epsilon l(\underline{Q})/4$.
By \eqref{eqq.partition'}, \eqref{eqq.partition_g_k} and \eqref{eqq.Q_underline} we obtain
\begin{equation}
\notag
g_{k+1}(x) < 1 \quad \hbox{for all} \quad x \in Q_{\delta_{\epsilon}}(\underline{x}) \cap \partial Q_{2}.
\end{equation}
Consequently, it follows from \eqref{eqq.psi} that $\psi_{Q'}(x) > 0$ for all $x \in Q_{\delta_{\epsilon}}(\underline{x}) \cap \partial Q_{2}$. This observation
together with the inequality $l(Q') \le 2^{j}l(Q_{2})$ proves \eqref{eqq.Hausdorff_measure_estimate}.

As a result, taking into account $(\psi4)$ of Proposition \ref{Prop.psi_properties} we deduce \eqref{eqq.5.14''} in this case, i.e.
\begin{equation}
\begin{split}
&\sum\limits_{Q' \in \mathcal{C}_{\mathcal{Q}}(Q_{1},Q_{2})}(l(Q'))^{n} \lesssim  \sum\limits_{Q' \in \mathcal{C}_{\mathcal{Q}}(Q_{1},Q_{2})}\mathcal{H}^{n}(\operatorname{supp}\psi_{Q'} \cap \partial Q_{2}) \lesssim  \mathcal{H}^{n}(\partial Q_{2}) \lesssim (l(Q_{2}))^{n}.
\end{split}
\end{equation}

\textit{Step 2.} In the second case we have $Q_{2} \in \mathcal{SH}_{\mathcal{Q}}(Q_{1})$. We claim that
\begin{equation}
\label{eqq.6.20}
2^{-j}l(Q_{2}) \le l(Q') \le 2^{j}l(Q_{2}) \quad \hbox{for every} \quad Q' \in \mathcal{C}_{\mathcal{Q}}(Q_{1},Q_{2}).
\end{equation}
The upper bound in \eqref{eqq.6.20} is an easy consequence of \eqref{eqq.contacting_family}.
To prove the lower bound assume on the contrary that there is $Q' \in \mathcal{C}_{\mathcal{Q}}(Q_{1},Q_{2})$ with $l(Q') < 2^{-j}l(Q_{2})$. By \eqref{eqq.contacting_family} this implies that
$Q' \in \widetilde{\mathcal{D}}_{+}(S)$,
$\partial Q' \cap \partial Q_{2} \neq \emptyset$ and $\operatorname{int}Q' \cap \operatorname{int}Q_{2} = \emptyset$. Since $Q' \in \mathcal{IC}_{\mathcal{Q}}(Q_{1})$, there
is a cube $K \in \mathcal{SH}_{\mathcal{Q}}(Q_{1})$ such that $K \subset Q'$. As a result, we get $l(K) < 2^{-j}l(Q_{2})$ and $\operatorname{dist}(K,Q_{2}) < 2^{-j}l(Q_{2})$.
This gives a contradiction with $(\text{SH.1})$ of Proposition \ref{Prop.shadow_cubes_comparable}.

Using \eqref{eqq.6.20} we obtain $\mathcal{C}_{\mathcal{Q}}(Q_{1},Q_{2}) \subset \operatorname{n}_{S}(Q_{2})$. Clearly, this gives \eqref{eqq.5.14''} in the case under consideration, i.e.
\begin{equation}
\notag
\sum\limits_{Q' \in \mathcal{C}_{\mathcal{Q}}(Q_{1},Q_{2})}(l(Q'))^{n} \le \sum\limits_{\operatorname{n}_{S}(Q_{2})}(l(Q'))^{n} \lesssim (\min\{l(Q_{1}),l(Q_{2})\})^{n}.
\end{equation}

The proof is complete.

\end{proof}

We say that a family $\mathcal{F} \subset \mathcal{Q}$ has \textit{depth} at most $N$, $N \in \mathbb{N}$, and write $\operatorname{depth}\mathcal{F} \le N$,
if, for every $Q_{1},Q_{2} \in \mathcal{F}$, the inclusion
$Q_{1} \subset Q_{2}$ implies $l(Q_{2}) \le 2^{jN}l(Q_{1})$.
\begin{Prop}
\label{Prop.interaction_multiplicity}
For each $Q \in \mathcal{Q}$,
\begin{equation}
\label{eqq.523}
\operatorname{MULT}(\{\theta K:K \in \mathcal{I}_{\mathcal{Q}}(Q) \hbox{ and } l(K) \le l(Q)\}) < +\infty.
\end{equation}
\end{Prop}

\begin{proof}
It will be convenient to split the proof into several steps. We put
\begin{equation}
\label{eqq.524}
\mathcal{F}:=\{K \in \mathcal{I}_{\mathcal{Q}}(Q):l(K) \le 2^{-2j}l(Q)\}
\end{equation}

\textit{Step 1.} Since $\mathcal{I}_{\mathcal{Q}}(Q) \subset \widetilde{\mathcal{D}}_{+}(S)$, by \eqref{eqq.dyadic_S_2} and Lemma \ref{Lm.covering_multiplicity} it is sufficient to show that
\begin{equation}
\label{eqq.525}
\operatorname{MULT}(\{\theta K:K \in \mathcal{F}\}) < +\infty.
\end{equation}

\textit{Step 2.} We show that the depth of the family $\mathcal{F}$ is at most 2. Indeed, assume the contrary. Then, there are cubes $K_{1},K_{2} \in \mathcal{F}$ such that $K_{2} \supset K_{1}$
and $l(K_{2}) \geq 2^{3j}l(K_{1})$. There exists a cube $Q_{1} \in \mathcal{IC}_{\mathcal{Q}}(Q)$ such that $\operatorname{n}_{S}(Q_{1}) \cap \mathcal{IC}_{\mathcal{Q}}(K_{1}) \neq \emptyset$. This implies that $l(Q_{1}) \le 2^{j}l(K_{1})$.
Hence, $l(K_{2}) \geq 2^{2j}l(Q_{1})$. Since $K_{1} \cap Q_{1} \neq \emptyset$, there is 
$Q_{2} \in \mathcal{REL}_{S}(K_{2})$ containing $Q_{1}$ with $l(Q_{2}) > l(Q_{1})$. By Definition \ref{Def.regular_family} there is $Q_{3} \in \mathcal{Q}$ such that
$l(Q_{3}) \le 2^{j}l(K_{2}) \le 2^{-j}l(Q)$, $Q_{3} \supset Q_{1}$ and $l(Q_{3}) > l(Q_{1})$. This contradicts to Definition \ref{Def.shadow_iceberg}.

\textit{Step 3.} Now if $\theta K_{1} \cap \theta K_{2} \neq \emptyset$ and $l(K_{2}) \geq 2^{4j}l(K_{1})$ for some $K_{1},K_{2} \in \mathcal{F}$, then there is $Q_{1} \in \mathcal{REL}_{S}(K_{2})$ containing 
$Q_{1}$ such that $l(Q_{1}) \geq 2^{-j}l(K_{2})$.
Hence, there is $Q_{2} \supset Q_{1}$ such that $2^{3j}l(K_{1}) \le l(Q_{2}) \le 2^{j}l(K_{2})$. Hence, $\operatorname{depth}\mathcal{F} \geq 3$.
It remains to apply Lemma \ref{Lm.covering_multiplicity}.


\end{proof}

\subsection{Fine properties of the extension}
We start with a relatively simple observation.
Keeping in mind Definition \ref{Def.abstract_extension} and \eqref{eqq.maximal_function} we can rewrite \eqref{eqq.abstract_extension}
\begin{equation}
\label{eqq.abstract_extension'}
\operatorname{Ext}_{\mathcal{Q}}[f](x):=\sum\limits_{Q' \in \widetilde{\mathcal{D}}_{+}(S)}a_{Q'}\psi_{Q'}(x), \quad x \in \mathbb{R}^{n+1},
\end{equation}
where $a_{Q'}=M^{S}_{\theta Q}[f]$ provided that $Q' \in \mathcal{IC}_{\mathcal{Q}}(Q)$ for some $Q \in \mathcal{Q}$.

\begin{Lm}
\label{Lm.auxiliary}
There is a constant $C=C(j,n,\gamma) > 0$ such that, for each $f \in L_{1}^{\rm loc}(\mathcal{H}^{n}\lfloor_{S})$,
\begin{equation}
\label{eqq.5.23'}
\int\limits_{\mathbb{R}^{n+1} \setminus U_{\frac{1}{2^{2j}}}(S)}\gamma(x)\|\nabla \operatorname{Ext}_{\mathcal{Q}}[f](x)\|\,dx
\le C \mathcal{LN}_{\gamma}[f].
\end{equation}
\end{Lm}

\begin{proof}
We split the proof int several steps.

\textit{Step 1.} Note that if $Q \in \widetilde{\mathcal{D}}_{k}(S)$ and $\operatorname{supp}\psi_{Q} \cap \mathbb{R}^{n+1} \setminus U_{\frac{1}{2^{2j}}}(S) \neq \emptyset$, then $k\in \{0,1,2\}$.
Indeed, if $k > 2$, then $l(Q) \le 2^{-3j}$, $\operatorname{supp}\psi_{Q} \subset (1+\epsilon)Q$ and $\underline{\theta}Q \cap S \neq \emptyset$. Hence, by \eqref{eqq.theta} and $(\text{D.5.2})$, $(\text{D.5.3})$
we would get $\operatorname{supp}\psi_{Q} \subset U_{\frac{1}{2^{2j}}}(S)$. As a result,
\begin{equation}
\label{eqq.5.24'}
\int\limits_{\mathbb{R}^{n+1} \setminus U_{\frac{1}{2^{2j}}}(S)}\gamma(x)\|\nabla \operatorname{Ext}_{\mathcal{Q}}[f]\|\,dx
\le \sum\limits_{k=0}^{2}\sum\limits_{Q \in \widetilde{\mathcal{D}}_{k}(S)}\int\limits_{(1+\epsilon)Q}\gamma(x)\|\nabla \operatorname{Ext}_{\mathcal{Q}}[f](x)\|\,dx.
\end{equation}

\textit{Step 2.} By \eqref{eqq.psi_gradient} and \eqref{eqq.abstract_extension'} (the corresponding constant depends only on $n,j$ and $\operatorname{C}_{\psi}$)
\begin{equation}
\label{eqq.5.25'}
\|\nabla \operatorname{Ext}_{\mathcal{Q}}[f](x)\| \lesssim  \sum\limits_{k=0}^{2}
\sum\limits_{Q' \in \widetilde{\mathcal{D}}_{k}(S)}\chi_{\operatorname{supp}\psi_{Q'}}(x)|a_{Q'}| \quad \hbox{for all} \quad x \in \mathbb{R}^{n+1} \setminus U_{\frac{1}{2^{2j}}}(S).
\end{equation}

\textit{Step 3.}
We recall that $\widetilde{\mathcal{D}}_{0}(S) \subset \mathcal{Q}$. Since $Q' \in \mathcal{IC}_{\mathcal{Q}}(Q')$ for every $Q' \in \mathcal{Q}$, we have
$a_{Q'} = M^{S}_{\theta Q'}[f]$ for every $Q' \in \widetilde{\mathcal{D}}_{0}(S)$. If $Q' \in \bigcup_{i=0}^{2}\widetilde{\mathcal{D}}_{i}(S)$, then either $a_{Q'} = M^{S}_{\theta Q'}[f]$, or
$a_{Q'} = M^{S}_{\theta K}[f]$
for the unique cube $K \in \widetilde{\mathcal{D}}_{0}(S)$ containing $Q'$. In the later case $\mathcal{H}^{n}(\theta Q' \cap S) \approx \mathcal{H}^{n}(\theta K \cap S)$. Hence, if $Q' \in \widetilde{\mathcal{D}}_{k}(S)$, $k \in \{0,1,2\}$, then taking into account Remark \ref{Rem.massive_part} we can estimate
\begin{equation}
\label{eqq.5.26'}
|a_{Q'}| \lesssim M^{S}_{\theta K}[|f|] \lesssim \int\limits_{\theta K \cap S}|f(x)|\,d\mathcal{H}^{n}(x),
\end{equation}
where $K \supset Q'$ (possibly $K=Q'$) and $K \in \widetilde{\mathcal{D}}_{0}(S)$.  As a result, given $K \in \widetilde{\mathcal{D}}_{0}(S)$,
\begin{equation}
\label{eqq.5.36'''}
\int\limits_{(1+\epsilon)K \setminus U_{\frac{1}{2^{2j}}}(S)}\gamma(x)\|\nabla \operatorname{Ext}_{\mathcal{Q}}[f](x)\|\,dx  \lesssim \gamma_{(1+\epsilon)K}\sum\limits_{\substack{K' \in \widetilde{\mathcal{D}}_{0}(S) \\
\theta K' \cap \theta K \neq \emptyset}}\int\limits_{\theta K' \cap S}|f(x)|\,d\mathcal{H}^{n}(x).
\end{equation} 

\textit{Step 4.} Combining \eqref{eqq.5.24'}, \eqref{eqq.5.36'''} and taking into account Lemma \ref{Lm.doubling} we deduce 
\begin{equation}
\begin{split}
\label{eqq.5.26'}
&\int\limits_{\mathbb{R}^{n+1} \setminus U_{\frac{1}{2^{2j}}}(S)}\gamma(x)\|\nabla \operatorname{Ext}_{\mathcal{Q}}[f]\|\,dx \lesssim \sum\limits_{K \in \widetilde{\mathcal{D}}_{0}(S)}\gamma_{K}\sum\limits_{\substack{K' \in \widetilde{\mathcal{D}}_{0}(S) \\
\theta K' \cap \theta K \neq \emptyset}}\int\limits_{\theta K' \cap S}|f(x)|\,d\mathcal{H}^{n}(x) \\
&\lesssim \sum\limits_{K \in \widetilde{\mathcal{D}}_{0}(S)}\gamma_{K}\int\limits_{\theta K \cap S}|f(x)|\,d\mathcal{H}^{n}(x). 
\end{split}
\end{equation}
Now \eqref{eqq.5.23'} follows from \eqref{eqq.5.26'} in combination with Proposition \ref{Prop.discretization_of_Lebesgue} applied with $c=\theta$.

The proof is complete.

\end{proof}

The following observation is an easy consequence of Definition \ref{Def.admissible_for_weight} and Definition \ref{Def.contacting_family}.
\begin{Lm}
\label{Lm.forgotten}
If $\mathcal{Q}$ is regular for $\gamma$, then $\underline{\gamma}_{K} \approx \underline{\gamma}_{Q}$ for each $Q \in \mathcal{Q}$
and every $K \in \mathcal{I}_{\mathcal{Q}}(Q)$. The corresponding constants depend only on $\operatorname{C}_{\gamma}$, $\operatorname{C}_{\operatorname{A}}$
and $n$.
\end{Lm}

Now we are ready to state the main result of this section.
\begin{Th}
\label{Th.extension_main}
If $\mathcal{Q}$ is admissible for $\gamma$, then there is a constant $C=C(j,n,\gamma) > 0$ such that, for each $f \in L_{1}^{\rm loc}(\mathcal{H}^{n}\lfloor_{S})$,
\begin{equation}
\label{eqq.main_estimate}
\int\limits_{\mathbb{R}^{n+1}}\gamma(x)\|\nabla \operatorname{Ext}_{\mathcal{Q}}[f](x)\|\,dx \le C \Bigl(\mathcal{LN}_{\gamma}[f] + \sum\limits_{Q \in \mathcal{Q}}\underline{\gamma}_{Q}\mathcal{E}^{S}_{4Q}[f]\Bigr).
\end{equation}
\end{Th}

\begin{proof}
To the end of the proof we fix an arbitrary $f \in L_{1}^{\rm loc}(\mathcal{H}^{n}\lfloor_{S})$ with $\mathcal{N}_{\gamma,q}[f] < +\infty$. In the case
$\mathcal{N}_{\gamma,q}[f] = +\infty$ inequality \eqref{eqq.main_estimate} is trivial.
It will be convenient to spit the rest of the proof into several natural steps.

\textit{Step 1.} By Lemma \ref{Lm.auxiliary} it is sufficient to establish existence of a constant $C > 0$ such that
\begin{equation}
\label{eqq.main_estimate'}
\operatorname{I}:=\int\limits_{U_{\frac{1}{2^{2j}}}(S)}\gamma(x)\|\nabla \operatorname{Ext}_{\mathcal{Q}}[f](x)\|\,dx \lesssim \sum\limits_{Q \in \mathcal{Q}}\underline{\gamma}_{Q}\mathcal{E}^{S}_{4Q}[f].
\end{equation}

\textit{Step 2.}
By $(\psi1)$ of Proposition \ref{Prop.psi_properties} it is easy to see that for each $Q \in \widetilde{\mathcal{D}}_{k}(S)$ with $k \geq 1$,
\begin{equation}
\notag
\sum_{k=0}^{\infty}\sum_{Q' \in \widetilde{\mathcal{D}}_{k}(S)}\nabla \psi_{Q'}(x)=0 \quad \hbox{for every} \quad x \in \operatorname{supp}\psi_{Q}.
\end{equation}
Hence, for each $Q \in \mathcal{Q}$, we get
\begin{equation}
\label{eqq.abstract_extension_derivative}
\nabla \operatorname{Ext}_{\mathcal{Q}}[f](x):=\sum\limits_{Q' \in \widetilde{\mathcal{D}}_{+}(S)}(a_{Q'}-M^{S}_{\theta Q}[f])\nabla \psi_{Q'}(x) \quad \hbox{for every} \quad x \in \operatorname{supp}\Psi_{Q}.
\end{equation}

\textit{Step 3.} Given $Q \in \mathcal{Q}$, we set
\begin{equation}
\notag
\operatorname{I}_{Q}:=\int\limits_{\operatorname{supp}\Psi_{Q}}\gamma(x)\|\nabla \operatorname{Ext}_{\mathcal{Q}}[f](x)\|\,dx.
\end{equation}
It follows directly from $(\text{Por}1)$ of Proposition \ref{Prop.porous_property} and $(\psi1)$ of Proposition \ref{Prop.psi_properties} that
\begin{equation}
\begin{split}
\operatorname{I} \le \sum\limits_{Q \in \mathcal{Q} \setminus \mathcal{D}_{0}(S)}\operatorname{I}_{Q}.\\
\end{split}
\end{equation}

\textit{Step 3.} Fix an arbitrary $Q \in \mathcal{Q}$ with $l(Q) < 1$ and estimate $\operatorname{I}_{Q}$. Using Lemma \ref{Lm.nonzero_gradient} we get
\begin{equation}
\label{eqq.5.22}
\operatorname{I}_{Q} \le \sum\limits_{Q' \in \mathcal{C}_{\mathcal{Q}}(Q)}
\int\limits_{\operatorname{supp}\psi_{Q'}}\gamma(x)\|\nabla \operatorname{Ext}_{\mathcal{Q}}[f](x)\|\,dx.
\end{equation}
Now we recall Definition \ref{Def.contacting_family}, formula \eqref{eqq.abstract_extension_derivative} and take into account $(\psi5)$ of Proposition \ref{Prop.psi_properties}.
This allows to continue estimate \eqref{eqq.5.22} to get (we also use \eqref{eqq.neighborin_cubes})
\begin{equation}
\begin{split}
\label{eqq.5.23}
&\operatorname{I}_{Q} \lesssim \sum\limits_{Q' \in \mathcal{C}_{\mathcal{Q}}(Q)}\sum\limits_{Q'' \in \operatorname{n}_{S}(Q') \setminus \mathcal{IC}_{\mathcal{Q}}(Q)}\frac{\gamma_{Q'}}{l(Q'')}|a_{Q''}-M^{S}_{\theta Q}[f]|\\
&\lesssim \sum\limits_{Q' \in \mathcal{C}_{\mathcal{Q}}(Q)}\frac{\gamma_{Q'}}{l(Q')}\sum\limits_{Q'' \in \operatorname{n}_{S}(Q') \setminus \mathcal{IC}_{\mathcal{Q}}(Q)}|a_{Q''}-M^{S}_{\theta Q}[f]|.
\end{split}
\end{equation}
Using Lemma \ref{Lm.doubling}, Proposition \ref{Prop.main_geometric_ingredient} and Lemma \ref{Lm.forgotten} we derive
\begin{equation}
\begin{split}
\label{eqq.5.23}
&\operatorname{I}_{Q}
\lesssim  \sum\limits_{Q' \in \mathcal{C}_{\mathcal{Q}}(Q)}(l(Q'))^{n}\sum\limits_{Q'' \in \operatorname{n}_{S}(Q') \setminus \mathcal{IC}_{\mathcal{Q}}(Q)}\underline{\gamma}_{Q}|a_{Q''}-M^{S}_{\theta Q}[f]|\\
&\lesssim \sum\limits_{K \in  \mathcal{I}_{\mathcal{Q}}(Q)}\underline{\gamma}_{Q}|M^{S}_{\theta K}[f]-M^{S}_{\theta Q}[f]|\sum\limits_{Q' \in \mathcal{C}_{\mathcal{Q}}(Q,K)}(l(Q'))^{n}\\
&\lesssim \sum\limits_{K \in \mathcal{I}_{\mathcal{Q}}(Q)}(\min\{l(K),l(Q)\})^{n}\underline{\gamma}_{Q}|M^{S}_{\theta K}[f]-M^{S}_{\theta Q}[f]|.
\end{split}
\end{equation}

\textit{Step 4.} Using \eqref{eqq.5.23} and combining Remark \ref{Rem.interacting_cubes} with Lemma \ref{Lm.forgotten} it is easy to see that
\begin{equation}
\label{eqq.538}
\begin{split}
&\sum\limits_{Q \in \mathcal{Q} \setminus \mathcal{D}_{0}(S)}\operatorname{I}_{Q} \lesssim C \sum\limits_{Q \in \mathcal{Q}}\sum\limits_{\substack{K \in \mathcal{I}_{\mathcal{Q}}(Q)
\\ l(K) \le l(Q)}}(l(K))^{n}\underline{\gamma}_{Q}|M^{S}_{\theta K}[f]-M^{S}_{\theta Q}[f]|.
\end{split}
\end{equation}
Note that $\theta K \subset 4Q$ provided that $l(K) \le l(Q)$. Hence, by Proposition \ref{Prop.interaction_multiplicity} and Remark \ref{Rem.massive_part}
\begin{equation}
\label{eqq.539}
\begin{split}
&\sum\limits_{\substack{K \in \mathcal{I}_{\mathcal{Q}}(Q)
\\ l(K) \le l(Q)}}(l(K))^{n}|M^{S}_{\theta K}[f]-M^{S}_{\theta Q}[f]|\\
&\lesssim \sum\limits_{\substack{K \in \mathcal{I}_{\mathcal{Q}}(Q)
\\ l(K) \le l(Q)}}\int\limits_{\theta K}\fint\limits_{\theta Q}|f(x)-f(y)|\,d\mathcal{H}^{n}(y)\,d\mathcal{H}^{n}(x) \lesssim (l(Q))^{n}\mathcal{A}^{S}_{4 Q}[f].
\end{split}
\end{equation}
Finally, combining \eqref{eqq.538}, \eqref{eqq.539} and using Remark \ref{Rem.different_averagings} we arrive at \eqref{eqq.main_estimate'}
\begin{equation}
\sum\limits_{Q \in \mathcal{Q} \setminus \mathcal{D}_{0}(S)}\operatorname{I}_{Q} \lesssim \sum\limits_{Q \in \mathcal{Q}}\underline{\gamma}_{Q}(l(Q))^{n}\mathcal{A}^{S}_{4 Q}[f]
 \lesssim \sum\limits_{Q \in \mathcal{Q}}\underline{\gamma}_{Q}\mathcal{E}^{S}_{4Q}[f].
\end{equation}

The proof is complete.

\end{proof}


\section{Proofs of the main results}
Throughout this paragraph we fix the following data:

\begin{itemize}

\item[\((\text{D.6.1})\)] a number $n \in \mathbb{N}$ and a weight $\gamma \in A_{1}^{\rm loc}(\mathbb{R}^{n+1})$;

\item[\((\text{D.6.2})\)] a set $S \in \mathcal{ADR}^{n}(\mathbb{R}^{n+1})$;

\item[\((\text{D.6.3})\)] a parameter $j \geq \max\{\underline{j}(S,2),5\}$;

\item[\((\text{D.6.4})\)] a parameter $q > \underline{q}:=(\operatorname{C}_{\gamma}2^{n+1})^{2j}$.

\end{itemize}

\subsection{The direct trace theorem}
In this section we provide an appropriate lower bound for the trace norm in terms of the functionals introduced in Section 4.

\begin{Prop}
\label{Prop.Lebesgue_estimate}
There is a constant $C > 0$ such that, for each $f \in W_{1}^{1}(\mathbb{R}^{n+1},\gamma)|_{S}$,
\begin{equation}
\label{eqq.Lebesgue_estimate}
\mathcal{LN}_{\gamma}[f] \le C\|f\|_{W_{1}^{1}(\mathbb{R}^{n+1},\gamma)|_{S}}.
\end{equation}
\end{Prop}

\begin{proof}
Fix $f \in W_{1}^{1}(\mathbb{R}^{n+1},\gamma)|_{S}$ and $F \in W_{1}^{1}(\mathbb{R}^{n+1},\gamma)$ with $F|_{S}=f$. 
Given $Q \in \widetilde{\mathcal{D}}_{0}$, we set $M_{cQ}[F]:=M^{\mathcal{L}^{n+1}}_{cQ}[F]$.
Taking into account Remark \ref{Rem.massive_part}  and using Proposition \ref{Prop.Ziemer_estimate} we get
\begin{equation}
\notag
\int\limits_{cQ \cap S}\Bigl|f(x)-M_{cQ}[F]\Bigr|\,d\mathcal{H}^{n}(x) \lesssim \int\limits_{cQ}\|\nabla F(y)\|\,dy.
\end{equation}
Hence, using Lemma \ref{Lm.doubling}, and taking into account Remark \ref{Rem.dyadic_lattice_multiplicity} we obtain
\begin{equation}
\begin{split}
&\widetilde{\mathcal{LN}}_{\gamma,c}[f] \lesssim \sum\limits_{Q \in \widetilde{\mathcal{D}}_{0}(S)}\gamma_{Q}\int\limits_{cQ \cap S}\Bigl|f(x)-M_{cQ}[F]\Bigr|\,d\mathcal{H}^{n}(x)
+\sum\limits_{Q \in \widetilde{\mathcal{D}}_{0}(S)}\gamma_{Q}\int\limits_{cQ \cap S}|F(y)|\,dy\\
&\lesssim \sum\limits_{Q \in \widetilde{\mathcal{D}}_{0}(S)}\underline{\gamma}_{cQ}\int\limits_{cQ \cap S}\|\nabla F(y)\|\,dy+\sum\limits_{Q \in \widetilde{\mathcal{D}}_{0}(S)}
\underline{\gamma}_{cQ}\int\limits_{cQ \cap S}|F(y)|\,dy\\
&\lesssim \sum\limits_{Q \in \widetilde{\mathcal{D}}_{0}(S)}\int\limits_{cQ \cap S}\gamma(y)\|\nabla F(y)\|\,dy+
\sum\limits_{Q \in \widetilde{\mathcal{D}}_{0}(S)}
\int\limits_{cQ \cap S}\gamma(y)|F(y)|\,dy \lesssim \|F\|_{W_{1}^{1}(\mathbb{R}^{n+1},\gamma)}.
\end{split}
\end{equation}
Since $f$ and $F$ with $F|_{S}=f$ were chosen arbitrarily, by \eqref{eqq.2.23} we obtain \eqref{eqq.Lebesgue_estimate}.
\end{proof}

Now we are ready to establish the main result of this paragraph

\begin{Prop}
\label{Prop.upper_green}
There is a constant $C > 0$
such that, for each $f \in W_{1}^{1}(\mathbb{R}^{n+1},\gamma)|_{S}$,
\begin{equation}
\label{eq.upper_green}
\mathcal{BN}_{\gamma,q}[f] \le C \|f\|_{W_{1}^{1}(\mathbb{R}^{n+1},\gamma)|_{S}}.
\end{equation}
\end{Prop}

\begin{proof}
Fix arbitrary $f \in W_{1}^{1}(\mathbb{R}^{n+1},\gamma)|_{S}$ and $F \in W_{1}^{1}(\mathbb{R}^{n+1},\gamma)$ with $F|_{S}=f$.
Applying Proposition \ref{Prop.discretization_of_Besov_1} and then Proposition \ref{Prop.Ziemer_estimate} we get
\begin{equation}
\label{eqq.3.7}
\mathcal{BN}_{\gamma,q}[f] \lesssim \sum\limits_{k \in \mathbb{Z}}q^{k}\sum\limits_{Q \in \mathcal{G}^{k}}(l(Q))^{n}\mathcal{E}^{S}_{4 Q}[f] \lesssim
\sum\limits_{k \in \mathbb{Z}}q^{k}\sum\limits_{Q \in \mathcal{G}^{k}}\int\limits_{4 Q}\|\nabla F(y)\|\,dy.
\end{equation}
Given $k \in \mathbb{Z}$, we set $U_{k}:=\bigcup_{Q \in \mathcal{G}^{k}}4 Q$.
By Proposition \ref{Prop.green_multiplicity} and the inequality $q > \underline{q} > 1$, 
\begin{equation}
\label{eqq.3.8}
\begin{split}
&\sum\limits_{k \in \mathbb{Z}}q^{k}\sum\limits_{Q \in \mathcal{G}^{k}}\int\limits_{4 Q}\|\nabla F(y)\|\,dy \lesssim \sum\limits_{k \in \mathbb{Z}}q^{k}\int\limits_{U_{k}}\|\nabla F(y)\|\,dy \\
&\lesssim \sum\limits_{j \in \mathbb{Z}}\Bigl(\sum\limits_{k \le j}q^{k}\Bigr)\int\limits_{U_{j}\setminus U_{j+1}}\|\nabla F(y)\|\,dy \lesssim \sum\limits_{j \in \mathbb{Z}}q^{j}\int\limits_{U_{j}\setminus U_{j+1}}\|\nabla F(y)\|\,dy\\
&\lesssim \sum\limits_{j \in \mathbb{Z}}\int\limits_{U_{j}\setminus U_{j+1}}\gamma(y)\|\nabla F(y)\|\,dy \lesssim \int\limits_{\mathbb{R}^{n+1}}\gamma(y)\|\nabla F(y)\|\,dy.
\end{split}
\end{equation}
Combining \eqref{eqq.3.7} and \eqref{eqq.3.8} we complete the proof.
\end{proof}

\subsection{A nonlinear algorithm}
Now we should construct a special regular for $S$ family of dyadic cubes.
After that we apply an abstract machinery developed
in Section 5 to that family. Unfortunately, requirements
$(\text{R.1})$--$(\text{R.3})$ in Definition \ref{Def.regular_family} are insufficient
for the operator \eqref{eqq.abstract_extension} to be the right inverse for the trace operator.
It is the only place, where nonlinearity, i.e., dependence on $f$ comes into play. To make the corresponding
informal explanations mathematically correct we introduce the following concept.

\begin{Def}
\label{Def.extension_family}
We say that $\mathcal{Q} \in \mathfrak{R}(S)$ is an extension family if for each $(x,r) \in S \times (0,1]$
there is a cube $Q \in \mathcal{Q}$ with $l(Q) \le r$ such that $Q \ni x$.
\end{Def}

Now we are ready to establish the main result of this section.
\begin{Th}
\label{Th.traffic_light}
There is $\underline{q} > 1$ such that for each $q > \underline{q}$ the following holds:
\begin{itemize}
\item[\((1)\)]  if $\varlimsup_{r \to +0}\gamma_{r}(x) = +\infty$ for all $x \in S$, then $\mathcal{G}_{\gamma}(q)$ is an admissible for $\gamma$ regular for $S$ extension family;

\item[\((2)\)] if $\varlimsup_{r \to +0}\gamma_{r}(x) < +\infty$, then, for each $f$ satisfying $\mathcal{N}_{\gamma,q}[f] < +\infty$, there is an admissible for $\gamma$ regular for $S$ extension family $\mathcal{Q}_{\gamma}(q,f) \supset \mathcal{G}_{\gamma}(q)$ such that
  \begin{equation}
  \label{eqq.forgotten_inequality}
  \sum\limits_{Q \in \mathcal{Q}_{\gamma}(q,f)\setminus \mathcal{G}_{\gamma}(q)}\mathcal{E}_{4Q}[f] \le \mathcal{N}_{\gamma,q}[f].
  \end{equation}  
\end{itemize}
\end{Th}

\begin{proof}
Note that $\varlimsup_{r \to +0}\gamma_{r}(x) = +\infty$ for all $x \in S$ if and only if $S^{k}=S$ for all $k \in \mathbb{Z}$.

To prove $(1)$ we put $\mathcal{G}:=\mathcal{G}_{\gamma}(q)$. This family is regular for $S$ by Theorem \ref{Th.green_is_regular}. Note that if $Q \in \mathcal{G}^{k}$ for some $k \in \mathbb{Z}$, then $K=\overline{Q}_{\mathcal{G}} \in \mathcal{G}^{k-1}$. Hence, $\underline{\gamma}_{5Q} \in [q^{k},q^{k+1})$ and $\underline{\gamma}_{5K} \in [q^{k-1},q^{k})$. This proves that $\mathcal{G}$ satisfies
the requirement of Definition \ref{Def.admissible_for_weight} with $\operatorname{C} = q^{2}$.
It remains to show that $\mathcal{G}$ is an extension family. We fix a point $\underline{x} \in S$ and a cube $\underline{Q} \in \mathcal{D}_{0}(S)$ containing $\underline{x}$.
Let $\underline{k}$ be such that $\underline{Q} \in \mathcal{G}^{\underline{k}}$. Since $S^{k}=S$ for all $k \in \mathbb{Z}$, there is a sequence of cubes $\{\underline{Q}_{k}\}_{k \geq \underline{k}}$ such that $\underline{Q}_{k} \in \mathcal{G}^{k}$
and $\underline{x} \in \underline{Q}_{k}$ for every $k \geq \underline{k}$. This proves the required property.

To prove $(2)$ we recall Theorem \ref{Th.nonlinear} and fix an arbitrary strictly increasing sequence $\{k_{s}\}_{s=0}^{\infty} \subset \mathbb{N}_{0}$ such that
$k_{0}=0$ and
\begin{equation}
\label{eqq.red_cubes}
\sum\limits_{Q \in \widetilde{\mathcal{D}}_{k_{s}}(S)}\gamma_{Q}\mathcal{E}^{S}_{4 Q}[f] \le 2^{-s}\mathcal{N}_{\gamma,q}[f] \quad \hbox{for every} \quad s \in \mathbb{N}_{0}.
\end{equation}
We put $\widetilde{\mathcal{R}}^{s}_{\gamma}(q,f):=\widetilde{\mathcal{D}}_{k_{s}}(S)$, $s \in \mathbb{N}_{0}$, and set $\widetilde{\mathcal{R}}_{\gamma}(q,f):=\bigcup_{s \in \mathbb{N}_{0}}\widetilde{\mathcal{R}}^{s}_{\gamma}(q,f)$.
It is to verify that the family $\widetilde{\mathcal{R}}_{\gamma}(q,f)$ is regular for $S$. Furthermore, by the construction $\widetilde{\mathcal{R}}_{\gamma}(q,f)$ is an extension family.
Now we put $\mathcal{Q}_{\gamma}(q,f):=\mathcal{G}_{\gamma}(q) \cup \widetilde{\mathcal{R}}_{\gamma}(q,f)$. By Proposition \ref{Prop.union_of_regular_families} $\mathcal{Q}_{\gamma}(q,f)$ is regular for $S$.
Furthermore, by the construction it is an extension family and \eqref{eqq.forgotten_inequality} follows from \eqref{eqq.red_cubes}. 

It remains to verify that $\mathcal{Q}_{\gamma}(q,f)$ is admissible for $\gamma$. We set $\mathcal{R}_{\gamma}(q,f):=\widetilde{\mathcal{R}}_{\gamma}(q,f)\setminus \mathcal{G}_{\gamma}(q)$ and paint each cube in $\mathcal{R}_{\gamma}(q,f)$ red. One should consider two cases. In the first case both cubes $Q$ and $K=\overline{Q}_{\mathcal{Q}}$ are of the same color.
The case when both cubes are green follows from the arguments made above for the family $\mathcal{G}$.
If both cubes are red, then there is $k^{\ast} \in \mathbb{Z}$ such that $\underline{\gamma}_{5K},\underline{\gamma}_{5Q} \in [q^{k^{\ast}},q^{k^{\ast}+1})$
and the required inequality follows. If $K$ is red and $Q$ is green, then $\underline{\gamma}_{5K} \in [q^{k^{\ast}},q^{k^{\ast}+1})$ and 
$\underline{\gamma}_{5Q} \in [q^{k^{\ast}+1},q^{k^{\ast}+2})$ for some $k^{\ast} \in \mathbb{Z}$ because
there are no green cubes that are contained in $K$ and contain $Q$. By similar arguments, if $K$ is green and $Q$ is red, then $\underline{\gamma}_{5K},\underline{\gamma}_{5Q} \in [q^{k^{\ast}},q^{k^{\ast}+1})$
for some $k^{\ast} \in \mathbb{Z}$.

The proof is complete.
\end{proof}

\subsection{The inverse trace theorem}
We start with the result showing that an abstract extension being associated with a regular extension families becomes a ``true'' extension.
\begin{Th}
\label{Th.ext_is_right_inverse}
Let $f \in L_{1}^{\rm loc}(\mathcal{H}^{n}\lfloor_{S})$ and let $\mathcal{Q} \in \mathfrak{R}(S)$ be an extension family. If $F:=\operatorname{Ext}_{\mathcal{Q}}[f]$, then 
$F \in L^{\rm loc}_{1}(\mathbb{R}^{n+1})$
and
\begin{equation}
\label{eqq.6.6''}
\operatorname{J}_{r}(x):=\fint\limits_{Q_{r}(x)}|f(x)-F(y)|\,dy \to 0, \quad r \to +0 \quad \hbox{for $\mathcal{H}^{n}$-a.e.  $x \in S$}.
\end{equation}
\end{Th}

\begin{proof}
We split the proof into several steps.

\textit{Step 1.} Since $f \in L_{1}^{\rm loc}(\mathcal{H}^{n}\lfloor_{S})$, there is a set $E_{f} \subset S$ such that $\mathcal{H}^{n}(E_{f}) = 0$
and each  point $x \in S \setminus E_{f}$ is a Lebesgue point of $f$, i.e., $\lim_{r \to +0}\mathcal{J}_{r}(x)=0$ (see Corollary 1 in Section 1.7 in \cite{Evans}), where we set
\begin{equation}
\label{eqq.612''}
\mathcal{J}_{r}(x):=\fint\limits_{S \cap Q_{r}(x)}|f(x)-f(y)|\,d\mathcal{H}^{n}(y), \quad (x,r) \in S \times (0,2].
\end{equation}

\textit{Step 2.} Given $x \in S \setminus E_{f}$ and $r \in (0,2]$, we recall notation \eqref{eqq.dyadic_S}, \eqref{eqq.dyadic_S_2} and put
$$
\mathcal{F}^{k}_{r}(x):=\{Q \in \widetilde{\mathcal{D}}_{k}(S) \cap \mathcal{Q}:\operatorname{supp}\Psi_{Q} \cap Q_{r}(x) \neq \emptyset\}.
$$
Furthermore, we set $\operatorname{T}_{r}(x):=\{k \in \mathbb{Z}:\mathcal{F}^{k}_{r}(x) \neq \emptyset\}$ and split this index set into small and big parts, i.e. $\operatorname{T}^{s}_{r}(x):=\{k \in \operatorname{T}_{r}(x):2^{-kj} \le r\}$
and $\operatorname{T}^{b}_{r}(x):=\{k \in \operatorname{T}_{r}(x):2^{-kj} > r\}$.

\textit{Step 3.} Note that $S \cap Q$ is compact for each $Q \in \mathcal{Q}$. Hence, given $Q \in \mathcal{Q}$, we have
\begin{equation}
\label{eqq.inf_is_positive}
\inf\{l(Q'):Q' \in \mathcal{SH}_{\mathcal{Q}}(Q)\} > 0.
\end{equation}
Indeed, otherwise there is a sequence of points $\{x_{s}\} \subset \underline{\theta}Q \cap S$ and
cubes $\{Q_{s}\}$ such that $\underline{\theta}Q_{s} \ni x_{s}$ and $l(Q_{s}) \to 0$, $s \to \infty$. Since $Q \cap S$ is closed and $Q_{s} \subset Q$ for all $s \in \mathbb{N}$, 
this gives existence of $\underline{x} \in S \cap Q$ and a subsequence
$\{x_{s_{m}}\}$ converging to $\underline{x}$. Since $\mathcal{Q}$ is an extension family, there is $\underline{Q} \in \mathcal{Q}$ such that $\underline{Q} \ni \underline{x}$.
At the same time $\operatorname{dist}(\underline{Q},Q_{s_{m}}) \to 0$, $m \to \infty$ and $\lim_{m \to \infty}l(Q_{s_{m}})=0$. By $(\text{SH.1})$ of Proposition \ref{Prop.shadow_cubes_comparable}
we get a contradiction. This proves \eqref{eqq.inf_is_positive}. As a result,
\begin{equation}
\label{eqq.distance_positive}
\operatorname{dist}(\operatorname{supp}\Psi_{Q},S) > 0 \quad \text{for each} \quad Q \in \mathcal{Q}.
\end{equation}
Given $x \in S \setminus E_{f}$ and $r \in (0,1]$, we put  $\underline{k}_{r}(x):=\min\{k \in \mathbb{Z}: k \in \operatorname{T}_{r}(x)\}$. By \eqref{eqq.distance_positive}
\begin{equation}
\label{eqq.key_limit}
\lim\limits_{r \to +0}\underline{k}_{r}(x) = +\infty.
\end{equation}

By \eqref{eqq.key_limit}, given $x \in S \setminus E_{f}$ and $\varepsilon > 0$ we can find $\delta(\epsilon) > 0$ such that for each $r \in (0,\delta(\varepsilon))$
\begin{equation}
\label{eqq.J_key_estimate}
\mathcal{J}_{R}(x) < \frac{\varepsilon}{2} \quad \text{for every} \quad R \in (0,3R(x,r)),
\end{equation}
where $R(x,r):=2^{-j\underline{k}_{r}(x)}$.

\textit{Step 4.} Given $x \in S \setminus E_{f}$ and $r \in (0,1]$, we put $E_{Q}:=\operatorname{supp}\Psi_{Q} \cap Q_{r}(x)$ for every $Q \in \mathcal{Q}$. Using \eqref{eqq.abstract_extension} we deduce
\begin{equation}
\label{eqq.613''}
|f(x)-F(y)| \le \sum\limits_{Q \in \mathcal{Q}}\chi_{E_{Q}}(y)|f(x)-M^{S}_{\theta Q}[f]|.
\end{equation}
We use \eqref{eqq.613''} and take into account $(\text{Por.1})$ of Proposition \ref{Prop.porous_property}. This gives
\begin{equation}
\label{eqq.J_estimate}
\operatorname{J}_{r}(x) \le \sum\limits_{k \in \operatorname{T}_{r}(x)}\sum\limits_{Q \in \mathcal{F}^{k}_{r}(x)}\frac{\mathcal{L}^{n+1}(E_{Q})}{(2r)^{n+1}}\fint\limits_{\theta Q \cap S}|f(x)-f(y)|\,d\mathcal{H}^{n}(y).
\end{equation}

\textit{Step 5.} Given $x \in S \setminus E_{f}$, if $k \in \operatorname{T}^{s}_{r}(x)$, then $\theta Q \subset Q_{3r}(x)$ for every $Q \in \mathcal{F}^{k}_{r}(x)$.
Furthermore, $\mathcal{L}^{n+1}(E_{Q}) \le (l(Q))^{n+1}$ for each $Q \in \mathcal{Q}$. As a result, using \eqref{eqq.J_key_estimate} and taking into account definition of the
family $\mathcal{F}^{k}_{r}(x)$ and Remark \ref{Rem.dyadic_lattice_multiplicity} we obtain
\begin{equation}
\label{eqq.614''}
\begin{split}
&\Sigma^{s}_{r}(x):=\sum\limits_{k \in \operatorname{T}^{s}_{r}(x)}\sum\limits_{Q \in \mathcal{F}^{k}_{r}(x)}\frac{\mathcal{L}^{n+1}(E_{Q})}{(2r)^{n+1}}\fint\limits_{\theta Q \cap S}|f(x)-f(y)|\,d\mathcal{H}^{n}(y)\\
&\lesssim \sum\limits_{k \in \operatorname{T}^{s}_{r}(x)}\frac{1}{r2^{kj}}\mathcal{J}_{3r}(x) \lesssim \mathcal{J}_{3r}(x) \lesssim \frac{\varepsilon}{2} \quad \text{for every} \quad r \in (0,\delta(\varepsilon)).
\end{split}
\end{equation}

\textit{Step 6.}
Given $x \in S \setminus E_{f}$, if $k \in \operatorname{T}^{b}_{r}(x)$, then $\theta Q \subset K(Q) \subset Q_{3R(x,r)}(x)$ for every $Q \in \mathcal{F}^{k}_{r}(x)$.
Here $K(Q)$ is the cube centered in $x$ with side length $3l(Q)$. Hence, using \eqref{eqq.J_key_estimate} and taking into account Remark \ref{Rem.loc_finite_2} and obvious inclusions $E_{Q} \subset Q_{r}(x)$, $Q \in \mathcal{Q}$, we arrive at
\begin{equation}
\label{eqq.621''}
\begin{split}
&\Sigma^{b}_{r}(x):=\sum\limits_{k \in \operatorname{T}^{b}_{r}(x)}\sum\limits_{Q \in \mathcal{F}^{k}_{r}(x)}\frac{\mathcal{L}^{n+1}(E_{Q})}{(2r)^{n+1}}\fint\limits_{\theta Q \cap S}|f(x)-f(y)|\,d\mathcal{H}^{n}(y)\\
& \lesssim \frac{\varepsilon}{2}\sum\limits_{k \in \operatorname{T}^{b}_{r}(x)}\sum\limits_{Q \in \mathcal{F}^{k}_{r}(x)}\frac{\mathcal{L}^{n+1}(E_{Q})}{(2r)^{n+1}} \lesssim \frac{\varepsilon}{2} 
\quad \text{for every} \quad r \in (0,\delta(\varepsilon)).
\end{split}
\end{equation}

\textit{Step 7.} Combine \eqref{eqq.J_estimate}, \eqref{eqq.614''}, \eqref{eqq.621''}. Since $\varepsilon > 0$ was arbitrary, we deduce \eqref{eqq.6.6''}.
Finally, it remains to show that $F \in L_{1}^{\rm loc}(\mathbb{R}^{n+1})$. It is sufficient to verify that $F \in L_{1}(Q)$ for every $Q \in \mathcal{D}_{0}$.
It is obvious if $Q \cap S = \emptyset$ because $F \in C^{\infty}(\mathbb{R}^{n+1}\setminus S)$. If $Q \cap S \neq \emptyset$, then $\mathcal{H}^{n}(2Q \cap S) > 0$
and hence, there is $x \in 2Q \cap S$ such that $f(x) \in \mathbb{R}$ and $2Q \subset Q_{2}(x)$. Taking into account that $\operatorname{T}_{2}(x)=\operatorname{T}^{s}_{2}(x)$
using estimates made in \eqref{eqq.614''} we get $\operatorname{J}_{2}(x) \lesssim \mathcal{J}_{6}(x) < +\infty$. Hence, $F \in L_{1}(Q)$.

The proof is complete.
\end{proof}

In order to make next steps we need to have a some sort of removability result for Sobolev spaces. Probably, this is a folklore. Since it is difficult to present
a precise reference, we supply the result with a detailed proof. 
\begin{Prop}
\label{Prop.removability}
Let $F \in L_{1}^{\rm loc}(\mathbb{R}^{n+1}) \cap C^{\infty}(\mathbb{R}^{n} \setminus S)$ be such that $F \in W_{1}^{1}(\operatorname{int}Q \setminus S)$ for every cube $Q$ in $\mathbb{R}^{n+1}$. 
Let $S_{F}$ be the intersection of the set of all $\mathcal{L}^{n+1}$-Lebesgue points of $F$ and the set of all $\mathcal{H}^{n}\lfloor_{S}$-Lebesgue points of $F$, respectively. 
If $\mathcal{H}^{n}(S \setminus S_{F})=0$, then $F \in W_{1}^{1,\rm loc}(\mathbb{R}^{n+1})$.
\end{Prop}

\begin{proof}
By the well known characterization of Sobolev spaces (see Theorem 2.1.4 in \cite{Zi}) it is sufficient to show that, for each $i \in \{1,...,n+1\}$, the restrictions of $F$ to almost all lines parallel to the $i$-th coordinate axis are locally absolutely continuous. 
We split the proof of this fact into several steps.

\textit{Step 1.} Without loss of generality we may assume that $i=n+1$, identify $\mathbb{R}^{n}$ with a hyperplane
$\mathbb{R}^{n} \times \{0\}$ in $\mathbb{R}^{n+1}$ and write each $x \in \mathbb{R}^{n+1}$ as a pair $x=(x',x_{n+1})$ where $x'$ is the projection
of $x$ to $\mathbb{R}^{n}$. The symbol $\Pi$ denotes the corresponding projection map. Given $x \in \mathbb{R}^{n+1}$, by $L(x)$ we denote the line passing through $x$ parallel to the $(n+1)$-th axis. 

\textit{Step 2.} According to the Eilenberg inequality \cite{Haj} there is 
$S' \subset S_{F}$ with $\mathcal{H}^{n}(S_{F} \setminus S') = 0$ such that $\#(L(y) \cap S_{F}) < +\infty$ for every $y \in S'$.
By the Fubini theorem there is  $S'' \subset S'$ such that $\mathcal{H}^{n}(S' \setminus S'') = 0$
and the pointwise restriction $F|_{L(x)}$ of $F$ to $L(x)$ belongs to $W_{1}^{1, \rm loc}(L(x) \setminus S_{F})$ for every $x \in S''$.
Finally, given $x \in S''$ and $r > 0$, we put
\begin{equation}
\notag
G_{r}(x):=\fint\limits_{x_{n+1}-r}^{x_{n+1}+r}|F(x',t)-F(x',x_{n+1})|\,dt.
\end{equation}

\textit{Step 3.}  We will also assume that $\mathcal{H}^{n}(\Pi(S'')) > 0$, because otherwise
there is nothing to prove.  
We claim that there is a set $\underline{E} \subset \Pi(S'')$ such that $\mathcal{H}^{n}(\Pi(S'') \setminus \underline{E}) = 0$
and
\begin{equation}
\label{eqq.618'}
\varliminf\limits_{r \to +0}G_{r}(x) = 0 \quad \text{for each} \quad x \in \underline{S}:=\Pi^{-1}(\underline{E}).
\end{equation}
Assume the contrary. Then, there is a set $E \subset \Pi(S'')$ with $\mathcal{H}^{n}(E) > 0$ and there is a parameter 
$\kappa > 0$ such that for each $x' \in E$ there is $x \in \Pi^{-1}(x')$ satisfying $\varliminf_{r \to +0}G_{r}(x) \geq 2\kappa$.
We put $S(E):=\{x \in \Pi^{-1}(E): \varliminf_{r \to +0}G_{r}(x) \geq 2\kappa\}$. Given $\delta > 0$, let $E_{\delta} \subset E$ be the set of all
$x' \in E$ for each of which there is $x \in S(E)$ satisfying $G_{r}(x) \geq \kappa$ for all $r \in (0,\delta)$. Clearly, 
$\mathcal{H}^{n}(E_{\delta^{\ast}}) > 0$ for some $\delta^{\ast} > 0$.
Let $S_{\delta^{\ast}}$ be the corresponding subset of $S(E)$. Since $\Pi(S_{\delta^{\ast}})=E_{\delta^{\ast}}$, we have 
$\mathcal{H}^{n}(S_{\delta^{\ast}}) \geq \mathcal{H}^{n}(E_{\delta^{\ast}}) > 0$.

\textit{Step 4.}
By Theorem 1, section 1.7.1 in \cite{Evans} there is 
$\underline{x} \in S_{\delta^{\ast}}$ such that $\underline{x}$ is a density point of $S_{\delta^{\ast}}$ and $\underline{x}'$ is a density point of $E_{\delta^{\ast}}$.
Hence, taking into account \eqref{eqq.Ahlfors_regular} we can decrease $\delta^{\ast} > 0$ if necessary and get, for each $r \in (0,\delta^{\ast})$, (below $Q_{r}(\underline{x}')$ means the cube in $\mathbb{R}^{n}$ centered
in $\underline{x}'$ with side length $2r$)
\begin{equation}
\begin{split}
\label{eqq.620'}
&(2r)^{n} \geq \mathcal{H}^{n}(Q_{r}(\underline{x}') \cap E_{\delta^{\ast}}) \geq \frac{(2r)^{n}}{2} \geq \frac{1}{2C^{S}_{2}}\mathcal{H}^{n}(Q_{r}(\underline{x}) \cap S_{\delta^{\ast}})\\ 
&\geq \frac{1}{4C^{S}_{2}}\mathcal{H}^{n}(Q_{r}(\underline{x}) \cap S) 
\geq \frac{C^{S}_{1}}{4C^{S}_{2}}r^{n}. 
\end{split}
\end{equation}
Given $x=(x',x_{n+1}) \in Q_{r}(\underline{x}) \cap S$, we clearly have
\begin{equation}
\label{eqq.621'}
\fint\limits_{x_{n+1}-r}^{x_{n+1}+r}|F(x',t)-F(x',x_{n+1})|\,dt \le 2\fint\limits_{\underline{x}_{n+1}-2r}^{\underline{x}_{n+1}+2r}|F(x',t)-F(\underline{x})|\,dt + |F(\underline{x})-F(x)|.
\end{equation}
Hence, combining \eqref{eqq.620'}, \eqref{eqq.621'} and using the Fubini theorem, we deduce for all $r \in (0,\delta^{\ast})$, 
\begin{equation}
\label{eqq.622'}
\begin{split}
&\kappa \mathcal{H}^{n}(Q_{r}(\underline{x}) \cap S_{\delta^{\ast}}) \le 2\kappa C^{S}_{2}\mathcal{H}^{n}(Q_{r}(\underline{x}') \cap E_{\delta^{\ast}})\\
&\lesssim \int\limits_{Q_{r}(\underline{x}') \cap E_{\delta^{\ast}}}\fint\limits_{\underline{x}_{n+1}-2r}^{\underline{x}_{n+1}+2r}|F(x',t)-F(\underline{x})|\,dt\,d\mathcal{H}^{n}(x')+\int\limits_{Q_{2r}(\underline{x})\cap S}|F(\underline{x})-F(x)|\,d\mathcal{H}^{n}(x)\\
&\lesssim r^{n}\Bigl(\fint\limits_{Q_{2r}(\underline{x})}|F(y)-F(\underline{x})|\,dy+\fint\limits_{Q_{2r}(\underline{x})\cap S}|F(\underline{x})-F(x)|\,d\mathcal{H}^{n}(x)\Bigr).
\end{split}
\end{equation}
Since $\underline{x}$ is an $\mathcal{L}^{n+1}$-Lebesgue point of $F$ and  
simultaneously $\underline{x}$ is an $\mathcal{H}^{n}$-Lebesgue point of $F$, taking small enough $r \in (0,\delta^{\ast})$ in \eqref{eqq.622'} 
and using \eqref{eqq.620'} we arrive at a contradiction.

\textit{Step 5.}
Fix an arbitrary point $\underline{x} \in \underline{S}$. Let $\{a_{i}(\underline{x})\}_{i=1}^{N(\underline{x})}$, $N(\underline{x}) \in \mathbb{N}$, be the set  $L(\underline{x}) \cap S_{F}$ taken in the natural order.
Since $F$ is smooth and Sobolev on $L(\underline{x}) \setminus \{a_{i}(\underline{x})\}_{i=1}^{N(\underline{x})}$, there exist one sided limits $f(a_{i}(\underline{x})\pm 0) \in \mathbb{R}$ for each $i \in \{1,...N\}$.
By \eqref{eqq.618'} it is easy to see that
\begin{equation}
\label{eqq.one_sided_limits}
f(a_{i}(\underline{x})-0)=f(a_{i}(\underline{x}) + 0) \quad \text{for each} \quad i \in \{1,...,N(\underline{x})\}.
\end{equation}

The proof is complete.

\end{proof}

Now we prove the so called inverse trace theorem, i.e. we estimate the trace norm from above in terms of the special functionals introduced
earlier. We recall Theorem \ref{Th.traffic_light} and put $\mathcal{Q}:=\mathcal{G}_{\gamma}(q)$ if $\lim_{r \to +0}\underline{\gamma}_{r}(x)=+\infty$ for all $x \in S$. 
If $\lim_{r \to +0}\underline{\gamma}_{r}(x) < +\infty$ for some $x \in S$ we put $\mathcal{Q}:=\mathcal{Q}_{\gamma}(q,f)$.
\begin{Th}
\label{Th.inverse_trace_estimate}
Let $f \in \mathfrak{B}(S)$ be such that $\mathcal{N}_{\gamma,q}[f] < +\infty$. 
Then $F=\operatorname{Ext}_{\mathcal{Q}}[f] \in W_{1}^{1}(\mathbb{R}^{n+1},\gamma)$ and
there is a constant $C > 0$ such that 
\begin{equation}
\label{eqq.inverse_Lebesgue_estimate}
\|F\|_{W_{1}^{1}(\mathbb{R}^{n+1},\gamma)} \le C \mathcal{N}_{\gamma,q}[f] \quad \text{for every} \quad f \in L_{1}^{\rm loc}(\mathcal{H}^{n}\lfloor_{S}).
\end{equation}
\end{Th}

\begin{proof} 
We recall that $\mathcal{L}^{n+1}(S)=0$ due to $(\text{Por.1})$ Proposition \ref{Prop.porous_property}. 
Hence, combining Proposition \ref{Prop.discretization_of_Lebesgue_2}, Theorem \ref{Th.extension_main}, and 
taking into account construction of the family $\mathcal{Q}$ together with \eqref{eqq.forgotten_inequality} we deduce 
\begin{equation}
\label{eqq.6.25''}
\int\limits_{\mathbb{R}^{n+1}}\gamma(x)\|\nabla F(x)\|\,dx \lesssim \mathcal{N}_{\gamma,q}[f].
\end{equation} 
Applying Theorem \ref{Th.ext_is_right_inverse} in combination with Proposition \ref{Prop.removability} we conclude that $F \in W^{1, \rm loc}_{1}(\mathbb{R}^{n+1})$. 
By \eqref{eqq.6.25''} and the definition of the space $W_{1}^{1}(\mathbb{R}^{n+1},\gamma)$ it remains to show that
\begin{equation}
\label{eqq.inverse_Lebesgue_estimate'}
\operatorname{J}:=\int\limits_{\mathbb{R}^{n+1}}\gamma(x)|F(x)|\,dx \lesssim \mathcal{N}_{\gamma,q}[f].
\end{equation}
We clearly have
\begin{equation}
\operatorname{J} \le \int\limits_{\mathbb{R}^{n+1}}\gamma(x)\Bigl|F(x)-\fint\limits_{Q_{1}(x)}F(y)\,dy\Bigr|\,dx+\int\limits_{\mathbb{R}^{n+1}}\gamma(x)\Bigl|\fint\limits_{Q_{1}(x)}F(y)\,dy\Bigr|\,dx
=:\operatorname{J}_{1}+\operatorname{J}_{2}.
\end{equation}
Since $F\in W^{1, \rm loc}_{1}(\mathbb{R}^{n+1})$ we can use Lemma 7.16 from \cite{Gilbarg}, change
the order of integration and take into account $(\text{D}4)$ in Lemma \ref{Lm.doubling}. This gives
\begin{equation}
\begin{split}
&\operatorname{J}_{1} \lesssim
\int\limits_{\mathbb{R}^{n+1}}\gamma(x)\int\limits_{Q_{1}(x)}\frac{\|\nabla F(y)\|}{\|x-y\|^{n}}\,dy\,dx\\
&\lesssim \int\limits_{\mathbb{R}^{n+1}}\|\nabla F(y)\|\int\limits_{Q_{1}(y)}\frac{\gamma(x)}{\|x-y\|^{n}}\,dx \lesssim \int\limits_{\mathbb{R}^{n+1}}\gamma(y)\|\nabla F(y)\|\,dy.
\end{split}
\end{equation}
Keeping in mind $(\text{Por.2})$ in Proposition \ref{Prop.porous_property} we can estimate
\begin{equation}
\begin{split}
\label{eqq.inverse_Lebesgue_estimate'}
&\operatorname{J}_{2} \le \sum\limits_{Q \in \mathcal{D}_{0}}\int\limits_{Q}\gamma(x)\Bigl|\fint\limits_{Q_{1}(x)}F(y)\,dy\Bigr|\,dx \lesssim \sum\limits_{Q \in \mathcal{D}_{0}}\int\limits_{ Q}\gamma(y)\Bigl|\fint\limits_{Q_{1}(x)}F(y)\,dy-\fint\limits_{\Omega_{Q}}F(z)\,dz\Bigr|\,dx\\
&+\sum\limits_{Q \in \mathcal{D}_{0}}\underline{\gamma}_{Q}\fint\limits_{\Omega_{Q}}|F(z)|\,dz=:\operatorname{J}_{2,1}+\operatorname{J}_{2,2}.
\end{split}
\end{equation}
Given $Q \in \mathcal{D}_{0}$, $Q_{1}(x) \subset 3Q$ for each $x \in Q$. Using Lemma \ref{Lm.doubling} 
it is easy to deduce $\underline{\gamma}_{Q} \lesssim \underline{\gamma}_{3Q}$. Combination of Remark \ref{Rem.different_averagings}, Proposition \ref{Prop.Ziemer_estimate} and 
\eqref{eqq.6.25''} leads to
\begin{equation}
\begin{split}
&\operatorname{J}_{2,1} \lesssim \sum\limits_{Q \in \mathcal{D}_{0}}\underline{\gamma}_{Q}\fint\limits_{3Q}\fint\limits_{3Q}|F(y)-F(x)|\,dy\,dz\\ 
&\lesssim 
\sum\limits_{Q \in \mathcal{D}_{0}}\underline{\gamma}_{3Q}\int\limits_{3Q}\|\nabla F(y)\|\,dy \lesssim \int\limits_{\mathbb{R}^{n+1}}\gamma(y)\|\nabla F(y)\|\,dy \lesssim \mathcal{N}_{\gamma,q}[f].
\end{split}
\end{equation}
Finally, since $\fint_{\Omega}|F(z)|\,dz = 0$ for $Q \in \mathcal{D}_{0} \setminus \widetilde{\mathcal{D}}_{0}(S)$, application of Proposition \ref{Prop.discretization_of_Lebesgue} gives
\begin{equation}
\begin{split}
&\operatorname{J}_{2,2} \lesssim \sum\limits_{Q \in \widetilde{\mathcal{D}}_{0}(S)}\underline{\gamma}_{Q}\fint\limits_{\theta Q \cap S}|f(x)|\,d\mathcal{H}^{n}(x) \lesssim \mathcal{LN}_{\gamma}[f].
\end{split}
\end{equation}
Collecting all the above estimates we complete the proof.
\end{proof}

\subsection{Final arguments}
Now we are ready to prove the two main results of this paper formulated in the introduction.
We recall $(\text{D.6.1})-(\text{D.6.4})$.

\textit{Proof of Theorems \ref{Th.first_main} and \ref{Th.Michal's_problem}.} If $f \in W_{1}^{1}(\mathbb{R}^{n+1},\gamma)|_{S}$, then 
by Propositions \ref{Prop.Lebesgue_estimate} and \ref{Prop.upper_green} we have $\mathcal{N}_{\gamma,q}[f] < +\infty$. 
Conversely, assume that $f\in \mathfrak{B}(S)$ is such that $\mathcal{N}_{\gamma,q}[f] < +\infty$. We recall Theorem \ref{Th.traffic_light}.
Let $\mathcal{Q}:=\mathcal{G}_{\gamma}(q)$ in the case $\varlimsup_{r \to +0}\underline{\gamma}_{r}(x) = +\infty$ for all $x \in S$
and $\mathcal{Q}:=\mathcal{Q}_{\gamma}(q,f)$ otherwise.  According to Theorem \ref{Th.inverse_trace_estimate} 
we have $F:=\operatorname{Ext}_{\mathcal{Q}} \in W_{1}^{1}(\mathbb{R}^{n+1},\gamma)$. 
By Theorem \ref{Th.ext_is_right_inverse} we have $F|_{S}=f$. Consequently, $f \in W_{1}^{1}(\mathbb{R}^{n+1},\gamma)|_{S}$. 
Furthermore, Propositions  \ref{Prop.Lebesgue_estimate}, \ref{Prop.upper_green} and Theorem \ref{Th.inverse_trace_estimate} 
allow to deduce that $\|f\|_{W_{1}^{1}(\mathbb{R}^{n+1},\gamma)|_{S}} \approx \mathcal{N}_{\gamma,q}[f]$, the equivalence constants being independent on $f$. 
This proves (1) in Theorem \ref{Th.first_main}.

Consider a mapping $W_{1}^{1}(\mathbb{R}^{n+1},\gamma)|_{S} \ni f \to \operatorname{Ext}_{\mathcal{Q}} \in W_{1}^{1}(\mathbb{R}^{n+1},\gamma)$ 
and denote it by the symbol $\operatorname{Ext}_{S,\gamma}$. By Theorem \ref{Th.ext_is_right_inverse} this map is an extension operator. 
In the case when $\varlimsup_{r \to +0}\underline{\gamma}_{r}(x) < +\infty$ for some $x \in S$ 
the map $\operatorname{Ext}_{S,\gamma}$ is nonlinear because the family $\mathcal{Q}$ depends on $f$. If $\varlimsup_{r \to +0}\underline{\gamma}_{r}(x) = +\infty$ for all $x \in S$, then
the family $\mathcal{Q}$ depends only on the geometry of $S$ and behaviour of $\gamma$. Hence, the family $\{\Psi_{Q}\}_{Q \in \mathcal{Q}}$ does not depend on $f$.
Furthermore, the map $M^{S}_{\theta Q}:L^{\rm loc}_{1}(\mathcal{H}^{n}\lfloor_{S}) \to \mathbb{R}$ is linear for each $Q \in \mathcal{Q}$. Now the linearity 
of $\operatorname{Ext}_{S,\gamma}$ is a direct consequence of \eqref{eqq.abstract_extension} and Remark \ref{Rem.extension_operator_well_defined}.

The proof is complete.

\hfill$\Box$

\end{document}